\documentclass[reqno,twoside,12pt]{article}
 \usepackage[english]{babel}
 \usepackage{times}
 \usepackage[
    paper=a4paper,
    portrait=true,
    textwidth=425pt,
    textheight=650pt,
    tmargin=3cm,
    marginratio=1:1
            ]{geometry}
 \usepackage[all]{xy}       %(Needed for drawing diagrams of functions)%
 \usepackage[centertags]{amsmath}
 \usepackage{amsfonts}
 \usepackage{amsthm}
 \usepackage{amssymb}
 \usepackage{enumerate}
 \usepackage{mathrsfs}      %(\mathcal: normal script; \mathscr: fancy script)%
 \usepackage{color}
 \usepackage{arydshln}      %(Needed for dashed lines in arrays)%
 \theoremstyle{plain}
 \newtheorem{Thm}{Theorem}[section]
 \newtheorem{Cor}[Thm]{Corollary}
 \newtheorem{Lemma}[Thm]{Lemma}
 \newtheorem{Prop}[Thm]{Proposition}

 \theoremstyle{definition}
 \newtheorem{Rem}[Thm]{Remark}
 
 \newtheorem{Defi}[Thm]{Definition}
 \numberwithin{Thm}{section}
 \numberwithin{equation}{section}

\def\text#1{\;\;\;\;{\rm \hbox{#1}}\;\;\;\;}
\def\qquad{\quad\quad}

\def\msy#1{{\mathbb #1}}
\def\C{{\msy C}}
\def\N{{\msy N}}

\def\R{{\msy R}}

\def\ga{\alpha}

\def\gd{\delta}

\def\gf{\varphi}

\def\gs{\sigma}

\def\gS{\Sigma}

\def\fa{{\mathfrak a}}
\def\fb{{\mathfrak b}}
\def\fg{{\mathfrak g}}
\def\fh{{\mathfrak h}}

\def\fk{{\mathfrak k}}
\def\fl{{\mathfrak l}}

\def\fn{{\mathfrak n}}

\def\fp{{\mathfrak p}}
\def\fq{{\mathfrak q}}
\def\fs{{\mathfrak s}}

\def\fu{{\mathfrak u}}

\def\to{\rightarrow}

\def\inp#1#2{\langle#1\,,\,#2\rangle}

\def\Ad{\mathrm{Ad}}

\def\ad{\mathrm{ad}}

\def\tr{\mathrm{tr}\,}

\def\ns{\fu_{k,l}}
\newcommand{\Ns}[1][{k,l}]{U_{#1}}
\def\hSigma{\overline{\Sigma}}

\newcommand{\Nq}[1][{k}]{V_{#1}}

\def\iq{{\mathrm q}}

\def\cA{{\mathcal A}}

\def\cC{{\mathcal C}}

\def\cH{{\mathcal H}}

\def\cP{{\mathcal P}}
\def\cQ{{\mathcal Q}}

\def\faq{\fa_\iq}
\def\faqd{\fa_\iq^*}

\def\fad{\fa^*}

\def\Lie{\mathop{\rm Lie}}

\def\Cartan{\theta}

\def\dotvar{\, \cdot\,}

\def\Ht{\cH}

\def\diag{\mathrm{diag}}

\def\Cen{Z}
\def\Nor{N}

\def\1{\mathbf{1}}

\def\cusp{\mathrm{cusp}}
\def\ds{\mathrm{ds}}

\def\cPH{\cP_\fh}
\def\cQH{\cQ_\fh}

\def\If{I_f}
\def\Iphi{I_\phi}
\def\firstJ{J}
\def\secJ{{}^c\!J}
\def\mathspec#1{\mathrm{#1}}
\def\SL{\mathspec{SL}}
\def\GL{\mathspec{GL}}
\def\Mat{\mathspec{Mat}}
\def\Aut{\mathspec{Aut}}
\def\Special{\mathspec{S}}
\def\SO{\mathspec{SO}}
\def\fsl{{\fs\fl}}

 \title{The notion of cusp forms for a class of reductive symmetric spaces of split rank one}
 \author{
    Erik~P.~van den Ban,
    Job~J.~Kuit\footnote{Supported by the Danish National Research Foundation through the Centre for Symmetry and Deformation (DNRF92).}\,\,
    and Henrik~Schlichtkrull
    }
\date{}%{January 8, 2016}

\begin{document}
 \maketitle
%\tableofcontents
%%% ----------------------------------------------------------------------
%%% ----------------------------------------------------------------------
\begin{abstract}
We study a notion of cusp forms for the symmetric spaces $G/H$ with $G=\SL(n,\R)$ and $H=\Special\big(\GL(n-1,\R)\times \GL(1,\R)\big)$. We classify all minimal parabolic subgroups of $G$ for which the associated cuspidal integrals are convergent and discuss the possible definitions of cusp forms. Finally, we show that the closure of the direct sum of the discrete series of representations of $G/H$ coincides with the space of cusp forms.
\end{abstract}
%%% ----------------------------------------------------------------------
%%% ----------------------------------------------------------------------
\section*{Introduction}\addcontentsline{toc}{section}{Introduction}
\setcounter{section}{9}
\setcounter{Thm}{0}
\setcounter{equation}{0}
\renewcommand{\thesection} {\Alph{section}}
%%% ----------------------------------------------------------------------
In this article we investigate the convergence of certain integrals that can be used to give a notion of cusp forms on the symmetric space $\SL(n,\R)/\Special\big(\GL(n-1,\R)\times\GL(1,\R)\big)$, which we here denote by
$X_{n}$.  Furthermore, we
 determine the relation between the discrete series representations and the space of cusp forms for these
spaces.

Harish-Chandra defined a notion of cusp forms for reductive Lie groups and proved
that the space of cusp forms coincides with the closed span in the $L^2$-Schwartz space 
of the discrete series of representations. This fact plays an important role in his work
on the Plancherel decomposition.
In \cite{AndersenFlenstedJensenSchlichtkrull_CuspidalDiscreteSerieseForSemisimpleSymmetricSpaces} a notion of cusp forms
for real hyperbolic spaces was introduced, following a more general suggestion by M.\ Flensted-Jensen. Subsequently, in
\cite{vdBanKuit_HC-TransformAndCuspForms}, the first and second author
gave a definition of cusp forms for
split rank $1$ reductive symmetric spaces. However,
the notion of \cite{vdBanKuit_HC-TransformAndCuspForms}
deviates from the general suggestion of Flensted--Jensen at an important point.
The main purpose of the present article is to explore
the necessity of  this deviation.

In order to give a precise description of the purpose of the present article we first recall
some background.
Let $G$ be semisimple and let $G/H$ be a symmetric space of split rank $1;$ here  $H$ is an open subgroup of the group of fixed points of an involution $\gs$ of $G.$
Every minimal parabolic subgroup $P$ of $G$ contains a $\sigma$-stable maximal split connected abelian subgroup $A_{P}$ of $G$. The set of minimal parabolic subgroups decomposes into two disjoint sets: the set $\cP$ of $P$ such that $A_{P}/(A_{P}\cap H)$ is $1$ dimensional and the set $\cQ$ of $P$ such that $A_{P}/(A_{P}\cap H)$ is $0$ dimensional, i.e., $A_{P}\subseteq H$.

The main goal is now to identify a suitable class of minimal parabolic subgroups $P$
with the property that for every $\phi$ in the
Harish-Chandra Schwartz space $\cC(G/H)$ the integral
\begin{equation}\label{eq int}
\int_{N_{P}/(N_{P}\cap H)}\phi(n)\,dn
\end{equation}
is absolutely convergent. Here $N_{P}$ denotes the unipotent radical of $P$ and $dn$ is an $N_{P}$-invariant Radon measure 
on $N_{P}/(N_{P}\cap H)$.
A cusp form is then defined to be a function $\phi\in\cC(G/H)$ such that
\begin{equation}\label{cusp form}
\int_{N_{P}/(N_{P}\cap H)}\phi(gn)\,dn=0
\end{equation}
for every such parabolic subgroup $P$ and every $g\in G$.

Flensted-Jensen has suggested to use the set of parabolic subgroups
that can be characterized as follows,
$$
\cP_{*}: =\{P\in\cP:\dim(N_{P}\cap H)=\max_{Q\in\cP}\dim(N_{Q}\cap H)\}.
$$
These parabolic subgroups are said to be $\fh$-extreme, see \cite[Def.\ 1.1]{vdBanKuit_HC-TransformAndCuspForms} for an equivalent characterization.

In \cite{AndersenFlenstedJensenSchlichtkrull_CuspidalDiscreteSerieseForSemisimpleSymmetricSpaces} it
was confirmed for real hyperbolic spaces that
the integral is absolutely convergent for $P\in\cP_{*}$ and
for every $\phi\in\cC(G/H)$.

In \cite{vdBanKuit_HC-TransformAndCuspForms} a notion called $\fh$-compatibility was introduced for parabolic subgroups $P\in\cP$ by imposing a condition on the roots that are positive for $P$. For the spaces $X_{n}$ this condition is recalled in Definition \ref{d: fh compatible}.
It was proved in \cite{vdBanKuit_HC-TransformAndCuspForms} that for $\fh$-compatible parabolic subgroups $P\in\cP$ the integrals (\ref{eq int}) are absolutely convergent.
Let
$$
\cPH := \{P\in\cP: P\mbox{ is  $\fh$-compatible}\}.
$$
For real hyperbolic spaces this set equals $\cP_{*}$, but
in general this need not be the case. The difference occurs for example for the symmetric spaces $X_{n}$.
If $n\geq 4$ then $\cP_{*}$ is not contained in $\cPH$, and if $n\geq3$ is odd then
$\cPH$ is not contained in $\cP_{*}$.
Therefore, the family $X_n$
is a good test-case for determining whether
or not the need for $\fh$-compatible parabolic subgroups in \cite{vdBanKuit_HC-TransformAndCuspForms}
is an artefact of the proof, and whether or not the dimension of $N_{P}\cap H$
is relevant for
the convergence of the integrals.

In Section \ref{Section Parabolic subgroups} we describe some generalities  concerning
 parabolic subgroups. The results in this section hold for any reductive symmetric space. In Section \ref{Section Symmetric space under consideration} we describe the polar decomposition of $X_{n}$ and the Harish-Chandra Schwartz space $\cC(X_{n})$ of $X_{n}$. Our main results are formulated in Section \ref{Section Main theorems} and proved
  in the remaining sections.

The first main result (Theorem \ref{Theorem main theorem convergence}) is a classification of all minimal parabolic subgroups $P$ such that the integral (\ref{eq int}) is absolutely convergent for all $\phi\in\cC(X_{n})$.
We extend the notion of $\fh$-compatibility to all minimal parabolic subgroups $P$ (not just those in $\cP$, but also the ones in $\cQ$).  We then show that (\ref{eq int}) is absolutely convergent for all $\phi\in\cC(X_{n})$ if and only if $P$ is $\fh$-compatible.

The second main result (Theorem \ref{Thm main theorem limit behavior}) describes the behavior at infinity of the function
\begin{equation}\label{eq def H_P}
\Ht_{P}\phi:A_{P}\to\C;\qquad a\mapsto a^{\rho_{P}}\int_{N_{P}/(N_{P}\cap H)}\phi(an)\,dn,
\end{equation}
for $P\in\cPH\cap\cP_{*}$ and $\phi\in\cC(X_{n})$. In particular it is shown that if $n$ is even, then $\Ht_{P}\phi$ is rapidly decreasing; if $n$ is odd, then $\Ht_{P}\phi$ is rapidly decreasing in one direction, while it converges to a possibly non-zero limit in the other direction. Moreover, for every $\fh$-compatible $P\in\cQ$ there exists a $P'\in\cPH\cap \cP_{*}$ and a $g\in G$ such that the integral (\ref{eq int}) equals a limit of $\Ht_{P'}\big(\phi(g\,\cdot\,)\big)$.

As a consequence of  Theorem \ref{Thm main theorem limit behavior}, we will show in Proposition \ref{Prop alternatives for def cusp form} that if (\ref{cusp form}) holds for all $P\in\cPH \cap \cP_*$,
then (\ref{cusp form}) holds for all $\fh$-compatible parabolic subgroups. This and Theorem \ref{Theorem main theorem convergence} justifies the use of $\cPH \cap \cP_*$ in the definition of cusp forms for reductive symmetric space of split rank $1$.

In \cite{AndersenFlenstedJensenSchlichtkrull_CuspidalDiscreteSerieseForSemisimpleSymmetricSpaces}
it was shown that there exist discrete series representations for certain real hyperbolic
spaces for which the generating functions are not cusp forms. These discrete series representations
are called non-cuspidal. Using the results of \cite{vdBanKuit_HC-TransformAndCuspForms} and the estimates from Theorem \ref{Thm main theorem limit behavior} for the behavior at infinity of (\ref{eq def H_P}), we show that for $X_{n}$ there are no non-cuspidal discrete series representations. Our final main result (Theorem \ref{Thm C_ds=C_cusp}) is thus that
the space of cusp forms on $X_{n}$ coincides with the closed span of the discrete series of representations for $X_{n}$.

For other papers concerning the particular symmetric space $X_n$, see for example
\cite{Kosters_EigenspacesOfTheLaplaceBeltramiOperator}, \cite{vDijkKosters_SphericalDistributions},
\cite{vDijkPoel_ThePlancherelFormula}, \cite{vDijkPoel_TheIrreducibleUnitaryGLSphericalRepresentationsOfSLn} and \cite{Ochiai_InvariantDistributions}.

We would like to thank Mogens Flensted-Jensen for many fruitful discussions related to the present work.
%%% ----------------------------------------------------------------------
\setcounter{section}{0}
\renewcommand{\thesection} {\arabic{section}}
\section{Parabolic subgroups and split components}\label{Section Parabolic subgroups}
%%% ----------------------------------------------------------------------
In this preliminary chapter we collect some properties which are valid for general reductive symmetric spaces
$G/H$. Here $G$ is a reductive Lie group of the Harish-Chandra class
and $H$ is an open subgroup of the group of fixed points for an involution $\sigma$ of $G$.
We are concerned with properties of $\sigma$-stable connected split abelian subgroups and
parabolic subgroups of $G$. The main result is that every parabolic subgroup $P$ of $G$ contains a
$\sigma$-stable maximal split abelian subgroup $A$ of $G$, which is unique up to conjugation by an
element of $N_{P}\cap H$.

%%% ----------------------------------------------------------------------
\subsection{Split abelian subalgebras}
%%% ----------------------------------------------------------------------

We write $\fg=\fh\oplus\fq$ for the eigenspace decomposition
for the infinitesimal involution $\sigma$. Here $\fh$ is the Lie algebra of $H$ and $\fq$ is the $-1$ eigenspace. Recall that an abelian subspace $\fa$ of $\fg$ is called split if $\fg$ decomposes as a sum of joint $\fa$-weight spaces.

\begin{Lemma}\label{Lemma existence theta commuting with sigma, stabilizing A}
Let $\fa$ be a $\sigma$-stable maximal split abelian subalgebra of $\fg$. Then there exists a $\sigma$-stable Cartan decomposition $\fg=\fk\oplus\fp$ with $\fa\subseteq\fp$.
\end{Lemma}

\begin{proof}
From the construction in \cite[Thms.~6.10, 6.11]{Knapp_LieGroupsBeyondAnIntroduction} it follows
that there exists a Cartan involution $\theta_0$ such that $\theta_0(X)=-X$ for $X\in\fa$.
Since $\fa$ is $\sigma$-stable, the construction in the proof of \cite[Thm.\ III.7.1]{Helgason_DGandSS},
applied to $\theta_0$ and $\sigma,$ gives a Cartan involution $\theta$ which in addition commutes with $\sigma$.
\end{proof}

\begin{Defi}
We define $\cA_{\fq}$ to be the set of all maximal split abelian subspaces $\fb$ of $\fq$, and
$\cA$ to be the set of all maximal split abelian subalgebras $\fa$ of $\fg$
for which $\fa\cap\fq\in\cA_{\fq}$.
The {\it split rank} of $G/H$ is the dimension of
any $\fb\in\cA_{\fq}$ (it will follow from Corollary \ref{Cor: split rank well defined} below that
this is well defined).
\end{Defi}

Note that $H$ acts on $\cA_{\fq}$ and $\cA$ by conjugation.

\begin{Lemma}\label{Lemma aq subset a implies a sigma-stable}
If $\fa\in\cA$, then $\fa$ is $\sigma$-stable.
\end{Lemma}

\begin{proof}
Let $Y\in \fa$. Then
$Y$ and $\sigma Y$ both commute with $\fa\cap\fq$ and hence $Y-\sigma Y\in\fa\cap\fq$ by
maximality. Hence, $\sigma Y\in\fa$.
\end{proof}

\begin{Prop}\label{characterization of cA}
A subspace $\fa\subseteq\fg$ belongs to $\cA$ if and only if
there exists a $\sigma$-stable Cartan decomposition $\fg=\fk\oplus\fp$ such that
$\fa$ is maximal abelian in $\fp$ and $\fa\cap\fq$ is maximal abelian in $\fp\cap\fq$.
The action of $H$ on $\cA$ by conjugation is transitive.
\end{Prop}

\begin{proof} Let $\cA'$ denote the set of
subspaces $\fa\subseteq\fg$ for which
there exists a $\sigma$-stable Cartan decomposition $\fg=\fk\oplus\fp$
such that $\fa$ is maximal abelian in $\fp$ and $\fa\cap\fq$ is maximal abelian
in $\fp\cap\fq$.
It follows from \cite[Lemmas 4,7]
{Matsuki_TheOrbitsOfAffineSymmetricSpacesUnderTheActionOfMinimalParabolicSubgroups}
that
$H$ acts transitively on $\cA'$.

It follows from Lemma \ref{Lemma existence theta commuting with sigma, stabilizing A} and Lemma \ref{Lemma aq subset a implies a sigma-stable} that $\cA\subseteq\cA'$. Hence, $\cA=\cA'$.
\end{proof}

\begin{Cor}\label{Cor: split rank well defined}
A subspace $\fb\subseteq\fq$ belongs to $\cA_{\fq}$ if and only if
there exists a $\sigma$-stable Cartan decomposition $\fg=\fk\oplus\fp$
such that $\fb$ is a maximal abelian subspace of
$\fp\cap\fq$.
The action of $H$ on $\cA_{\fq}$ by conjugation is transitive.
\end{Cor}

\begin{proof}
Let $\fb\in\cA_{\fq}$ and let
$\fa\in\cA$ with $\fb=\fa\cap\fq$.
The asserted Cartan decomposition
exists according to Proposition \ref{characterization of cA}.

Conversely, let $\fb\subseteq\fq$ and assume that $\fb$ is maximal abelian in $\fp\cap\fq$
for some $\sigma$-stable Cartan decomposition $\fg=\fk\oplus\fp$.
Let $\fa\subseteq\fp$ be a maximal abelian subspace with $\fa\supset\fb$.
Then $\fa\cap\fq=\fb$ by maximality. It follows from
Proposition \ref{characterization of cA} that $\fa\in\cA$
and hence $\fb\in\cA_\fq$.

The transitivity of the action follows from the corresponding statement in Proposition \ref{characterization of cA}.
\end{proof}

%%% ----------------------------------------------------------------------
\subsection{Parabolic subgroups}
%%% ----------------------------------------------------------------------

We recall that if $P$ and $Q$ are parabolic subgroups, then $Q$ is called opposite to $P$ if
$P\cap Q$ is a common Levi subgroup of $P$ and $Q$. If $P$ is a parabolic subgroup
we write $N_{P}$ for its unipotent radical.
Recall also that a {\it split component} of $P$ is a maximal connected split subgroup
of the center of a Levi subgroup of $P$.

Note that a parabolic subgroup $P$ is minimal if and only if the Lie algebras of its split components are
maximal split in $\fg$. Note also that if $\fa$ is a maximal split abelian subalgebra of $\fg$ and $A=\exp(\fa)$, then
the normalizer $N_G(A)$ of $A$ in $G$ acts by conjugation on the set of minimal parabolic subgroups containing $A$. This action is transitive.

\begin{Lemma}\label{Lemma A unique up to N conjugation}
Let $P$ be a minimal parabolic subgroup. If $A$ and $B$ are two split components of $P$, then there exists a unique $n\in N_{P}$ such that $B=nAn^{-1}$.
\end{Lemma}

\begin{proof}
There exists $g\in G$ such that $B=gAg^{-1}$. Then $P$ and $g^{-1}Pg$ both contain $A$ and hence
$g^{-1}Pg=wPw^{-1}$ for some $w\in N_G(A)$. The product $gw$ normalizes $P$
and thus belongs to $P$. The existence of $n$ now follows by decomposing this element according to the Levi decomposition $P = Z_G(A) N_P.$ 

Now assume that $n_{1},n_{2}\in N_{P}$ satisfy $B=n_{1,2}An_{1,2}^{-1}$. Then $n_{1}^{-1}n_{2}$
centralizes $A$ and thus $n_{1}^{-1}n_{2}\in Z_{G}(A)\cap N_{P}=\{e\}$. This proves uniqueness.
\end{proof}

\begin{Lemma}\label{Lemma common A unique up to N_1 cap N_2 conjugation}
Let $P$ and $Q$ be minimal
parabolic subgroups. Assume that $P$ and $Q$ have common split components $A$ and $B$. Then there exists a unique $n\in N_{P}\cap N_{Q}$ such that $B=nAn^{-1}$.
\end{Lemma}

\begin{proof}
By Lemma \ref{Lemma A unique up to N conjugation} there exists a unique $n\in N_{P}$ such that $B=nAn^{-1}$.
It suffices to show that $n\in N_{Q}$.
Let $a\in A$ be dominant with respect to $P$. Then for every $k\in \N$ we have $a^{-k}na^{k}n^{-1}\in a^{-k}B\subseteq Q$.
Since $Q$ is a closed subgroup and $a^{-k}na^{k}n^{-1}$ converges to $n^{-1}$ for $k\to\infty$, it follows that $n\in Q$.
Furthermore, since $n\in N_{P}$, the element $\log(n)$ is a sum of $\fa$-weight vectors
with non-zero weights. This implies that $n\in N_{Q}$.
\end{proof}

\begin{Thm}\label{Thm P contains sigma-stable A, unique up to N cap H conjugation}
Let $P$ be a minimal parabolic subgroup.
\begin{enumerate}[(i)]
\item
There exists a $\sigma$-stable split component $A$ of $P$.
\item
If $A$ and $B$ are two $\sigma$-stable split components of $P$, then there exists a unique $n\in N_{P}\cap H$ such that $B=nAn^{-1}$.
\end{enumerate}
\end{Thm}

\begin{proof}
See \cite[Lemma 12]{Rossmann_TheStructureOfSemisimpleSymmetricSpaces} for a proof that uses
\cite{Mostow_FullyReducibleSubgroupsOfAlgebraicGroups} and only applies to algebraic groups.
An alternative proof is given in \cite[Lemma 2.4]{HelminckWang_RationalityProperties},
also under the assumption of algebraic groups.
The statement {\em (i)} is proved without this assumption in \cite[Lemma 2]{Matsuki_TheOrbitsOfAffineSymmetricSpacesUnderTheActionOfMinimalParabolicSubgroups}, and hence
we only need to prove {\em (ii)}.

If $A$ and $B$ are two $\sigma$-stable split components of $P$,
then by Lemma \ref{Lemma common A unique up to N_1 cap N_2 conjugation} there exists a
unique $n\in N_{P}\cap\sigma(N_{P})$ such that $B=nAn^{-1}$. Observe that
$$
B
=\sigma(B)
=\sigma(n)\sigma(A)\sigma(n)^{-1}
=\sigma(n)A\sigma(n)^{-1}.
$$
From the uniqueness of $n$ we conclude that $\sigma(n)=n$, or equivalently, $n\in N_{P}\cap H$.
\end{proof}

The preceding theorem allows for the following definition.

\begin{Defi}\
Let $P$ be a minimal parabolic subgroup and let $A$ be a $\sigma$-stable split component of $P$. We define the $\sigma$-parabolic rank of $P$ to be the dimension of $A/(A\cap H)$.
\end{Defi}

We write $\cP$ for the set of minimal parabolic subgroups of maximal $\sigma$-parabolic rank. (Note that $P\in\cP$ if and only if $P$ admits a split component $A = \exp \fa$ with
$\fa \in\cA$.)
Furthermore, we write $\cQ$ for the set of all minimal parabolic subgroups of minimal $\sigma$-parabolic rank. If $\fa\in\cA$, then we write $\cP(\fa)$ for the set of minimal parabolic subgroups containing $\exp\fa$.
Note that $\cP(\fa)\subseteq\cP$. If $\fb$ is a $\sigma$-stable maximal split abelian subalgebra of $\fg$ such that $\fb\cap\fh$ has maximal dimension, then we write $\cQ(\fb)$ for the set of minimal parabolic subgroups containing $\exp\fb$. Note that $\cQ(\fb)\subseteq\cQ$.
\break

\begin{Cor}\label{Cor H-conjugacy of parabolic subgroups}\
\begin{enumerate}[(i)]
\item Let $\fa\in\cA$ and let $P\in\cP$. Then there exists an element $h\in H$ such that $hPh^{-1}\in\cP(\fa)$.
\item Let $\fb$ be a $\sigma$-stable maximal split abelian subalgebra of $\fg$ such that $\fb\cap \fh$ has maximal dimension. Then for every $Q\in\cQ$ there exists an element $h\in H$ such that $hQh^{-1}\in\cQ(\fb)$.
\end{enumerate}
\end{Cor}

\begin{proof}
{\em Ad (i):}
There exists a $\fb\in\cA$ such that $\exp\fb\subseteq P$. By Proposition \ref{characterization of cA} there exists an $h\in H$ such that $A=hBh^{-1}$. This implies that $A$ is contained in $hPh^{-1}$.

{\em Ad (ii):}
The claim follows from the fact that all such subalgebras are $H$-conjugate. See
Lemma \ref{Lemma existence theta commuting with sigma, stabilizing A} and
\cite[Lemmas 4,7]{Matsuki_TheOrbitsOfAffineSymmetricSpacesUnderTheActionOfMinimalParabolicSubgroups}.
\end{proof}

%%%%------------------------
\subsection{Positive systems of  $\fh$-roots}\label{subsection Positive systems of pure roots}
%%%-----------------------
Let $A$ be a $\sigma$-stable connected maximal split abelian subgroup of $G$ and let $\fa$ be its Lie algebra. We write $\Sigma(\fa)$ for the root system of  $\fa$ in $\fg.$ Given a root $\alpha\in\Sigma(\fa),$ we write $\fg_{\alpha}$ for the associated root space.

\begin{Defi}\label{Def pure h-roots}
By an $\fh$-root in $\Sigma(\fa)$ we mean a root  $\alpha\in\Sigma(\fa)$
such that $\fg_{\alpha}\cap \fh\neq\{0\}$.
The set of such roots is denoted by $\Sigma_{\fh}(\fa).$
\end{Defi}

\begin{Prop}\label{Prop properties Sigma_h(a)}\
\begin{enumerate}[(i)]
\item
If $\alpha\in\Sigma_\fh(\fa),$ then  $\alpha\big|_{\fa\cap\fq}=0$.
\item $\Cen_{\fh}(\fa\cap\fq)$ is a reductive Lie algebra of which
$\fa\cap \fh$ is a maximal split abelian subalgebra.
The root system of $\fa \cap \fh$ in
$\Cen_{\fh}(\fa\cap\fq)$ equals $\Sigma_\fh(\fa),$ viewed as a subset of
$(\fa \cap \fh)^*.$ 
\item Assume that $\fa\in\cA$. Then $\Sigma_{\fh}(\fa)=\Sigma(\fa)\cap(\fa\cap\fh)^{*}$. Furthermore, if $\alpha\in\Sigma_{\fh}(\fa)$, then $\fg_{\alpha}\subseteq\fh$.
\end{enumerate}
\end{Prop}

\begin{proof}
{\em Ad (i):} If $\fg_{\alpha}\cap \fh\neq\{0\}$, then $\fg_{\sigma \alpha}\cap \fg_{\alpha}=\sigma(\fg_{\alpha})\cap\fg_{\alpha}\neq\{0\}$. This implies $\sigma \alpha=\alpha$.

{\em Ad (ii):} Let $\theta$ be a Cartan involution giving rise to a Cartan decomposition as in Lemma \ref{Lemma existence theta commuting with sigma, stabilizing A}.
Since $\Cen_{\fh}(\fa\cap\fq)$ is $\theta$-stable, it follows from \cite[Cor.~6.29]{Knapp_LieGroupsBeyondAnIntroduction} that $\Cen_{\fh}(\fa\cap\fq)$ is reductive. The maximality of $\fa$ implies that $\fa\cap \fh$ is a maximal split abelian subalgebra of $\Cen_{\fh}(\fa\cap\fq)$. Let $\Phi$ be the root system of  $\fa\cap \fh$ in $\Cen_{\fh}(\fa\cap\fq).$  It follows from {\em (i)} that $\alpha|_{\fa\cap\fh}\in\Phi$ for every $\alpha\in\Sigma_{\fh}(\fa)$. Now let $\beta\in\Phi$ and let $\alpha\in\fa^{*}$ be given by $\alpha|_{\fa\cap\fh}=\beta$ and $\alpha|_{\fa\cap\fq}=0$. Then $\fg_{\alpha}\cap\fh$ contains the root space of $\beta$, hence $\alpha\in\Sigma_{\fh}(\fa)$. We conclude that $\Phi=\Sigma_{\fh}(\fa)$.

{\em Ad (iii):} It suffices to prove that under the given assumption, $\fg_{\alpha}\subseteq\fh$ for every  $\alpha\in\Sigma(\fa)$ that vanishes on $\fa\cap \fq$. Let $\alpha$ be such a root and let $X\in\fg_{\alpha}\cap\fq$. We write $\theta$ for a Cartan involution giving rise to a Cartan decomposition as in Lemma \ref{Lemma existence theta commuting with sigma, stabilizing A}. Since $X$ is in the centralizer of $\fa\cap \fq$, it follows from the maximality of $\fa\cap \fq$ that $X-\theta(X)\in\fa\cap\fq$. Since $X-\theta(X)\in\fg_{\alpha}\oplus\fg_{-\alpha}$, it follows that $X-\theta(X)=0$ and therefore $X=0$.
\end{proof}

If $P\in\cP(\fa)$, then we write
$\Sigma(\fa;P)$ for the positive system of $\Sigma(\fa)$ consisting of roots that are positive with respect to $P$. In other words, a root $\alpha\in\Sigma(\fa)$ is an element of $\Sigma(\fa;P)$ if and only if the root space $\fg_{\alpha}$ is contained in $\fn_{P}$.

We fix a positive system $\Sigma_{\fh}^{+}(\fa)$ of $\Sigma_{\fh}(\fa)$.

\begin{Cor}\label{Cor Q in cP(A) is Nor_(K cap H)(fa)-conjugate to ps with Sigma_h^+ subseteq Sigma(P)}
Let $Q$ be a minimal parabolic subgroup and let $A$ be a $\sigma$-stable split component of $Q$. Then $Q$ is $\Nor_{H}(\fa\cap \fh)\cap\Cen_{H}(\fa\cap\fq)$-conjugate to a minimal parabolic subgroup $P$ such that
$$
\Sigma_{\fh}^{+}(\fa)\subseteq\Sigma(\fa;P).
$$
\end{Cor}

\begin{proof}
The set $\Sigma_{\fh}(\fa)\cap\Sigma(\fa;Q)$ is a positive system for the root system $\Sigma_{\fh}(\fa)$. It follows from Proposition \ref{Prop properties Sigma_h(a)} {\em (ii)} that each such positive system is conjugate to $\Sigma_{\fh}^{+}(\fa)$ by an element of the normalizer of $\fa\cap\fh$ in $\Cen_{H}(\fa\cap\fq)$.
\end{proof}

%%% ----------------------------------------------------------------------
\section{The symmetric space under consideration}\label{Section Symmetric space under consideration}
%%% ----------------------------------------------------------------------
\subsection{The space $X_{n}$}\label{subsection the space Xn}
%%% ----------------------------------------------------------------------
For the remainder of this article $n$ will be a natural number with $n\geq 3$
 and $G$ will be the real Lie group $\SL(n,\R)$.
 Let $\sigma$ be the involution on $G$ given by
$$\sigma(g)=SgS^{-1},$$
where
$$
S
=S^{-1}
=\left(
          \begin{array}{c:c:c}
               &           &   -1\\
            \hdashline
                &   I_{n-2} &     \\
            \hdashline
            -1  &           &    \\
          \end{array}
        \right).
$$
The fixed point subgroup $H$ of $\sigma$ is obtained from the subgroup $\Special\big(\GL(n-1,\R)\times\GL(1,\R)\big)$ of $G$ (embedded in the usual manner) by conjugation with the orthogonal matrix
\begin{equation}\label{eq def kappa}
\kappa
=\left(
          \begin{array}{c:c:c}
            \frac{1}{\sqrt{2}}   &           &   \frac{1}{\sqrt{2}}\\
            \hdashline
                                &   I_{n-2} &                    \\
            \hdashline
            \frac{-1}{\sqrt{2}}  &           &   \frac{1}{\sqrt{2}} \\
          \end{array}
        \right).
\end{equation}
We denote the $2n-2$ dimensional reductive symmetric space $G/H$ by $X_{n}$.

Let $\theta$ be the Cartan involution given by
$$\theta(g)=(g^{-1})^{t},$$
and let $K$ be the fixed point subgroup of $\theta$, i.e., $K$ is the maximal compact subgroup $\SO(n)$.
Since $S$ is orthogonal, the involutions $\sigma$ and $\theta$ commute. The Lie algebra $\fg$ of $G$ admits the decomposition $\fg=\fh\oplus\fq$  into the eigenspaces
for $\gs.$
 Here
$$
\fh=
\big\{
\left(
          \begin{array}{c:c:c}
            \frac{-\tr x}{2}    &   v^{t}   &   y                   \\
            \hdashline
            w                   &   x       &   -w                  \\
            \hdashline
            y                   &   -v^{t}  &   \frac{-\tr x}{2}    \\
          \end{array}
        \right)
:x\in \Mat(n-2,\R), y\in\R, v\in\R^{n-2},w\in\R^{n-2}\big\}
$$
is the Lie algebra of $H$
 and
$$
\fq=
\big\{
\left(
          \begin{array}{c:c:c}
            z   &   v^{t}   &   y   \\
            \hdashline
            w   &           &   w   \\
            \hdashline
            -y  &   v^{t}   &   -z  \\
          \end{array}
\right)
:y\in\R, z\in\R, v\in\R^{n-2}, w\in\R^{n-2}\big\}.
$$
 Similarly,  we have the Cartan decomposition $\fg=\fk\oplus\fp$; here $\fk$ equals $\mathfrak{so}(n)$ and $\fp$ equals the space
of traceless symmetric $n \times n$ matrices.
We note that
$$
\fa_{\fq} :=
\R\, \diag(1,0,\dots,0,-1)
$$
is a maximal abelian subspace of $\fp\cap\fq$. Hence, the split rank of $X_{n}$ is equal to $1$.
We  put $A_{\fq}:= \exp(\fa_{\fq})$ and for $t\in\R$ we define $a_{t} \in A_{\fq}$  by
\begin{equation}\label{eq def a_t}
a_{t}:=\exp\big(t\,\diag(1,0,\dots,0,-1)\big).
\end{equation}

%%% ----------------------------------------------------------------------
\subsection{Polar decomposition}\label{subsection Polar decomposition}
%%% ----------------------------------------------------------------------
The space $X_{n}$ admits a polar decomposition: the map
$$
K\times A_{\fq}\to X_{n};\qquad (k,a)\mapsto ka\cdot H
$$
is surjective. If $a\cdot H\in K a'\cdot H$ for $a_,a'\in A_{\fq}$, then there exists a $k$ in $N_{K\cap H}(\faq),$ the normalizer of $\fa_{\fq}$ in $K\cap H,$ such that $a=ka'k^{-1}$. In fact, since the action of
$N_{K\cap H}(\faq)$ on $\fa_{\fq}$ is length-preserving (with respect to the Killing form) and
\begin{equation}\label{eq def k0}
k_{0}=
\left(
  \begin{array}{cc:c:c}
     0  &   0   &           &   1  \\
     0  &   -1  &           &   0   \\
    \hdashline
        &       &   I_{n-3} &       \\
    \hdashline
    1  &   0   &           &   0   \\
  \end{array}
\right)
\end{equation}
is an element in $K\cap H$ such that $k_{0}ak_{0}^{-1}=a^{-1}$ for every $a\in A_{\fq}$, we have
$a\cdot H\in K a'\cdot H$ if and only if $a'\in\{a,a^{-1}\}$.

\begin{Lemma}\label{Lemma KAH decomposition}
Let $g\in G$ and $t\in \R$. If $g\in Ka_{t}\cdot H$, then
$$
\|g\sigma(g)^{-1}\|_{HS}^{2}=n-2+2\cosh(4t),
$$
where $\|\cdot\|_{HS}$ denotes the Hilbert-Schmidt norm on $\Mat(n,\R)$.
\end{Lemma}

\begin{proof}
A straightforward computation shows that
$$
\|g\sigma(g)^{-1}\|_{HS}^{2}
=\tr\Big(g\sigma(g)^{-1}\big(g\sigma(g)^{-1}\big)^{t}\Big)
=\tr(a_{t}^{4}).
$$
The result now follows from the fact that $a_{t}^{4}$ equals the matrix $\diag(e^{4t},1,\dots,1,e^{-4t})$.
\end{proof}

%%% ----------------------------------------------------------------------
\subsection{Schwartz functions}
%%% ----------------------------------------------------------------------
In this section we give a definition of the space of Schwartz function on $X_{n}$. This definition differs from the one in  \cite[Sect.~17]{vdBan_PrincipalSeriesII}, but it is easily seen from \cite[Thm.~17.1, Prop.~17.2]{vdBan_PrincipalSeriesII}
combined with Remark \ref{r: calculation rho zero} of the present paper,
that the definitions are equivalent.

\begin{Defi}\label{defi cC}
A Schwartz function on $X_{n}$ is a smooth function $\phi:X_{n}\to\C$, such that for every $u\in U(\fg)$ and $m\geq0$ the seminorm
$$
\sup_{k\in K,t\in\R}\cosh^{n-1}(t)\big(1+ |t|)^m
\big|(u\phi)(ka_{t}\cdot H)\big|
$$
is finite. 
Here $u\phi$ is obtained with the regular action of $U(\fg)$ from the left. The vector space of Schwartz functions on $X_{n},$ equipped with
the locally convex topology induced by these seminorms, is denoted by $\cC(X_{n})$.
\end{Defi}

\begin{Rem}
\label{r: Schwartz seminorms}
To simplify computational expressions later on, it will be useful to work
with the following seminorms instead, for $u \in U(\fg)$ and $m \in \N,$
\begin{equation}
\label{e: seminorm mu u m}
\mu_{u,m}(\phi):=\sup_{k\in K,t\in\R}(2 \cosh 4t)^{\frac{n-1}{4}}\big(1+\log(2\cosh(4t))\big)^{m}\big|(u\phi)(ka_{t}\cdot H)\big|.
\end{equation}
Obviously, a smooth function $\phi: X_n \to \C$ belongs to $\cC(X_n)$
if and only if these seminorms are finite. Moreover, the seminorms $\mu_{u,m}$ detemine
the Fr\'echet topology on $\cC(X_n).$
\end{Rem}

For future reference we shall now construct some specific Schwartz functions on $X_n.$

\begin{Lemma}\label{Lemma phi_nu is Schwartz}
Let $\nu<\frac{1-n}{4}$. The function $\phi_{\nu}:X_{n}\to\R$, defined by
$$
\phi_{\nu}(ka_{t}\cdot H)=\cosh^{\nu}(4t),
$$
is Schwartz.
\end{Lemma}

\begin{proof}
We will show that $\cC(X_{n})$ contains every function $\phi\in C^\infty(X_n)$ such that
$$
K\times(1,\infty)\to\C;\qquad
(k,t)\mapsto\phi(ka_t\cdot H)
$$
is a linear combination of
functions of the form
\begin{equation}\label{qwerty}
K\times(1,\infty)\to\C;\qquad
(k,t)\mapsto \varphi(k)\cosh^\lambda(t)\sinh^\mu(t),
\end{equation}
with $\varphi\in C^\infty(K)$ and real numbers $\lambda,\mu$ with sum $\lambda+\mu=4\nu<1-n$.
Clearly our function $\phi_\nu$ has this property.
It is easily seen that $\mu_{1,m}(\phi)<\infty$ for such a function, and hence it suffices to
show that this class of functions is invariant under the left action by $U(\fg)$. For this
it suffices to consider the action by $\fg$. Let $X\in\fg$, then we can write $X$ as a linear
combination of elements of the form $\Ad(k)Y$ with some fixed basis elements $Y\in\fg$
and with coefficients which are smooth functions of $k\in K$.
We shall use a basis of weight vectors $Y$ for the adjoint action of $\fa_\fq$
on $\fg.$ 
If  $Y$ belongs to
$\fa_\fq$ or $\Cen_\fh(\fa_\fq)$, then it is easily seen that
$(k,t)\mapsto [\Ad(k)(Y)\phi](ka_t\cdot H)$ will again be
of the form (\ref{qwerty}) with $\lambda+\nu=4\nu$.
It remains to consider the case where $Y$ is a root vector for a root $\alpha$ of $\fa_\fq.$

Since 
$\fa_\fq$ is $\sigma\theta$-invariant, the associated root space in $\fg$ is $\sigma\theta$-invariant as well, hence we may assume that
$\sigma\theta Y=\pm Y$. If $\sigma\theta Y=Y$, a simple computation shows
that
$$Y=\frac{-a^{-\alpha}}{a^\alpha-a^{-\alpha}}(Y+\theta Y)+
\frac{1}{a^\alpha-a^{-\alpha}}\Ad(a)(Y+\sigma Y)$$
for all $a\in A_\fq$, and if $\sigma\theta Y=-Y,$ the corresponding formula reads
$$Y=\frac{a^{-\alpha}}{a^\alpha+a^{-\alpha}}(Y+\theta Y)+
\frac{1}{a^\alpha+a^{-\alpha}}\Ad(a)(Y+\sigma Y).$$
Note that for $a=a_t$, in both cases the coefficient of $Y+\theta Y$ is a linear combination
of functions $\cosh^p t\sinh^q t$ with $p+q=0$. Hence, when we apply
$\Ad(k)Y$ to $\phi$ as above, the term with $Y+\theta Y$ will produce a new function
of the same kind, whereas the term with $\Ad(a_t)(Y+\sigma Y)$ will be annihilated
because of the $H$-invariance from the right.
This proves the claim.
\end{proof}

%%% ----------------------------------------------------------------------
\subsection{$\fh$-compatible parabolic subgroups}\label{Subsection H-compatible parabolic subgroups}
%%% ----------------------------------------------------------------------
Let $P$ be a minimal parabolic subgroup. By Theorem \ref{Thm P contains sigma-stable A, unique up to N cap H conjugation} there exists a $\sigma$-stable split component $A$ of $P$, which is unique up to conjugation by $P\cap H$. We fix such a split component $A$ and write $\fa$ for its Lie algebra. We recall the definition of the root system $\Sigma_{\fh}(\fa)$ of $\fh$-roots from Definition \ref{Def pure h-roots}. We write $\Sigma_{\fh}(\fa;P)$ for the positive system $\Sigma(\fa;P)\cap\Sigma_{\fh}(\fa)$ of $\Sigma_{\fh}(\fa)$ and define
$$
\rho_{P,\fh}
=\frac{1}{2}\sum_{\alpha\in\Sigma_{\fh}(\fa;P)}\alpha.
$$

\begin{Defi}\label{d: fh compatible}
The parabolic subgroup $P$ is said to be $\fh$-compatible if one of the following conditions are fulfilled:
\begin{itemize}
\item[(a)]
    $P$ is of $\sigma$-parabolic rank $1$ and $\langle\alpha,\rho_{P,\fh}\rangle\geq 0$ for all $\alpha\in\Sigma(\fa;P)$;
\item[(b)]
    $P$ is of $\sigma$-parabolic rank $0$ and $\langle\alpha,\rho_{P,\fh}\rangle>0$ for all $\alpha\in\Sigma(\fa;P)$.
\end{itemize}
\end{Defi}

\begin{Rem}
\label{r: h comp for H conjugation}
We note that this notion is independent of the choice of $\fa,$ since
any other choice would be $P\cap H$-conjugate to $\fa.$
Furthermore, it is now readily seen that the property of $\fh$-compatibility is preserved
under conjugation by $H.$
\end{Rem}

Since $X_{n}$ is of split rank $1$, every minimal parabolic subgroup is either of $\sigma$-parabolic rank $1$ or of $\sigma$-parabolic rank $0$. We recall
that $\cP$ denotes the set of minimal parabolic subgroups of $\sigma$-parabolic rank $1$ and $\cQ$ denotes the set of minimal parabolic subgroups of $\sigma$-parabolic rank $0$. We write $\cPH$ and $\cQH$ for the sets of $\fh$-compatible parabolic subgroups in $\cP$ and $\cQ,$ respectively. Recall that $\cP_{*}$ denotes the set of $P\in\cP$ such that the dimension of $N_{P}/(N_{P}\cap H)$ is minimal, i.e.,  $\dim\big(N_{P}/(N_{P}\cap H)\big)=n-1$; see
Proposition \ref{Prop P H-conj to P_k,l; H-comp iff (n+1)/2 leq k,l leq (n+3)/2} and  Equation (\ref{eq n=(n cap h)+u}).

For every $n\geq 3$ the set $\cPH$ is non-empty. If $n$ is even, then $H$ acts transitively on $\cPH$ and $\cPH\subseteq\cP_{*}$. If $n$ is odd, then the
$H$-action admits three orbits, see Proposition \ref{Prop P H-conj to P_k,l; H-comp iff (n+1)/2 leq k,l leq (n+3)/2}. One orbit is equal to $\cPH\setminus(\cP_{*}\cap\cPH)$; the other two orbits are contained in $\cP_{*}\cap \cPH$.

If $n$ is even, then $\cQH=\emptyset$. If $n$ is odd, then $\cQH$ is non-empty and $H$ acts transitively on it, see Proposition \ref{Prop Q H-comp for b iff H-conj to Q_k with k=(n+1)/2 }.

%%% ----------------------------------------------------------------------
\section{Main theorems}\label{Section Main theorems}
%%% ----------------------------------------------------------------------

%%% ----------------------------------------------------------------------
\subsection{Convergence}\label{subsection main theorems Convergence}
%%% ----------------------------------------------------------------------
Let $P$ be a minimal parabolic subgroup of $G$ and let $dx$ be a non-zero $N_{P}$-invariant Radon measure on $N_{P}/(N_{P}\cap H)$.

\begin{Thm}\label{Theorem main theorem convergence}
The integral
$$
\int_{N_{P}/(N_{P}\cap H)}\phi(x\cdot H)\,dx
$$
is absolutely convergent for every $\phi\in\cC(X_{n})$ if and only if $P$ is $\fh$-compatible.
\end{Thm}

The proof of this theorem will be  given in Sections \ref{section Proof of Main theorem 1 for sigma parabolic rank 1} and \ref{section Proof of Main theorem 1 for sigma parabolic rank 0}.

%%% ----------------------------------------------------------------------
\subsection{Limit behavior}\label{subsection main theorems Limit behavior}
%%% ----------------------------------------------------------------------

Assume that $P\in\cPH\cap\cP_{*}$. Let $A$ be a $\sigma$-stable split component of $P$ and let $L$ be the centralizer of $A$ in $G.$ Then $L=MA$, where $M$
 is the (unique) maximal compact subgroup of $L$. Now $P=LN_{P}$ and $P=MAN_{P}$ are a Levi and a Langlands decomposition respectively. Note that $\fa=\Lie(A)\in\cA$. We define
$$
\delta_{P}:L \to\R_{>0};
\qquad l\mapsto\left|\frac{\det\Ad(l)\big|_{\fn_{P}}}{\det\Ad(l)\big|_{\fn_{P}\cap\Cen_{\fg}(\fa\cap\fq)}}\right|^{\frac{1}{2}}.
$$
Let
$$
\rho_{P}
=\frac{1}{2}\sum_{\alpha\in\Sigma(\fa;P)}\alpha.
$$
Then for $m\in M$ and $a\in A$
$$
\delta_{P}(ma)=a^{\rho_{P}-\rho_{P,\fh}}.
$$
For $\phi\in\cC(X_{n})$ we define its Harish-Chandra transform $\Ht_{P}\phi$ to be the function on $L/(L\cap H)$ given by
$$
\Ht_{P}\phi(l)= \delta_{P}(l)\int_{N_{P}/(N_{P}\cap H)} \phi(ln)\,dn\qquad(l\in L).
$$
Note that the integrals are absolutely convergent by Theorem \ref{Theorem main theorem convergence} and define a right $(L\cap H)$-invariant function on $L.$

\break

\begin{Thm}\label{Thm main theorem limit behavior}
Let $P \in \cPH \cap \cP_*$ be as above and let
$v\in\fa\cap \fq$ be such that $\rho_{P}(v)>0$.
\begin{enumerate}[(i)]
\item
    Assume $n$ is even. Then for every $\phi\in\cC(X_{n})$ and every $N\in\N,$
    $$
    \sup_{m\in M, t\in\R}\left|t^{N}\Ht_{P}\phi\big(m\exp(tv)\big)\right|<\infty.
    $$
\item
    Assume $n$ is odd.
   Then for every $\phi\in\cC(X_{n})$, every $r\in\R$ 
   and every $N\in\N,$ 
    $$
    \sup_{m\in M, t<r}\left|t^{N}\Ht_{P}\phi\big(m\exp(tv)\big)\right|<\infty.
    $$
    Moreover, the limit
        $$
    \mu_{P}(\phi)
    :=\lim_{t\to\infty}\Ht_{P}\phi\big(\exp(tv)\big)
    $$
    exists, and there exists an $\fh$-compatible  $Q\in\cQ$ and an element $g\in G$ such that for every $\phi\in\cC(X_{n}),$
    $$
    \mu_{P}(\phi)
    =\int_{N_{Q}/(N_{Q}\cap H)}\phi(gn)\,dn.
    $$
    Vice versa, if $Q\in\cQ$, then there exists a $P'\in\cPH\cap\cP_{*}$ and an element $g\in G$ such that for every $\phi\in\cC(X_{n}),$
    $$
    \int_{N_{Q}/(N_{Q}\cap H)}\phi(gn)\,dn
    =\mu_{P'}(\phi).
    $$
\end{enumerate}
\end{Thm}

The proof of the theorem will be given in Section \ref{Section Limits}.

%%% ----------------------------------------------------------------------
\subsection{Cusp forms}\label{subsection main theorems Cusp forms}
%%% ----------------------------------------------------------------------
As explained in the introduction, the aim of the article is to explore which parabolic subgroups should be used in the definition of cusp forms for reductive symmetric spaces of split rank $1$. In
\cite{vdBanKuit_HC-TransformAndCuspForms}
 it was proved that for such a symmetric space $X$ the integral
$$
\int_{N_{P}/(N_{P}\cap H)}\phi(n)\,dn
$$
is absolutely convergent for all $\phi\in\cC(X)$ and every $P\in\cPH$.

It follows from Theorem \ref{Theorem main theorem convergence} that for the spaces $X_{n}$  {\em only}  the $\fh$-compatible parabolic subgroups provide integrals that are convergent for all Schwartz functions. We conclude from this that the condition that $P$ is $\fh$-compatible, which was needed in \cite{vdBanKuit_HC-TransformAndCuspForms}, is not an artefact of the proof.

For $n$ odd the set $\cQH$ is non-empty. In \cite{vdBanKuit_HC-TransformAndCuspForms} only the minimal parabolic subgroups from 
$\cPH$ were used. The remaining question that needs to be answered is whether the class of parabolic subgroups that is used for the definition of cusp forms should include any minimal parabolic subgroup from $\cQH$.

\break

\begin{Prop}\label{Prop alternatives for def cusp form}
Let $\phi\in\cC(X_{n})$. The following four conditions are equivalent.
\begin{enumerate}[(i)]
\item For every $g\in G$ and every $\fh$-compatible parabolic subgroup $P,$
\begin{equation}\label{eq def cusp form}
\int_{N_{P}/(N_{P}\cap H)}\phi(gn)\,dn=0.
\end{equation}
\item For every $g\in G$ and every $P\in\cPH$ equation (\ref{eq def cusp form}) holds.
\item For every $g\in G$ and every $P\in\cPH\cap \cP_{*}$ equation (\ref{eq def cusp form}) holds.
\item\label{Prop cusp form def iv} There exists a $P$ in each $H$-conjugacy class in $\cPH\cap\cP_{*}$ such that for every $g\in G$ equation (\ref{eq def cusp form}) holds.
\end{enumerate}
\end{Prop}

\begin{proof}
The implications {\em (i)}$\;\Rightarrow\;${\em (ii)}$\;\Rightarrow\;${\em (iii)}$\;\Rightarrow\;${\em (iv)} are trivial.

Conversely, if  (\ref{eq def cusp form}) holds for a given $\fh$-compatible parabolic subgroup $P$ and every $g\in G$, then it also holds for every $H$-conjugate of $P$. This proves {\em (iv)}$\;\Rightarrow\;${\em (iii)}.
The implication {\em (iii)}$\;\Rightarrow\;${\em (ii)} is proven in
\cite[Lemma 8.14]{vdBanKuit_HC-TransformAndCuspForms}. 

Now assume $Q\in\cQH$. By Theorem \ref{Thm main theorem limit behavior} the integrals over $N_{Q}/(N_{Q}\cap H)$ can be obtained as limits of integrals over $N_{P}/(N_{P}\cap H)$ for some $P\in\cPH\cap\cP_{*}$. This establishes the implication {\em (ii)}$\;\Rightarrow\;${\em (i)}.
\end{proof}

\begin{Defi}
A function $\phi\in\cC(X_{n})$ is said to be a cusp form if
one of the equivalent conditions in Proposition \ref{Prop alternatives for def cusp form} is satisfied.
\end{Defi}

We write $\cC_{\cusp}(X_{n})$ for the space of cusp forms on $X_{n}$ and $\cC_{\ds}(X_{n})$ for the closed span of the discrete series representations of $X_{n}$. Given $g\in G$ and $\phi\in\cC(X_{n}),$ we write
$L_g \phi$ for the function given by $L_g \phi(x)=\phi(g^{-1}x),$
for $x\in X_{n}$.
In \cite{vdBanKuit_HC-TransformAndCuspForms} the following theorem, which we here only state for $X_{n}$, is proved for general reductive symmetric spaces of split rank $1$.

\begin{Thm}\label{Thm relation cusp forms and discrete series}\
\begin{enumerate}[(i)]
\item $\cC_{\cusp}(X_{n})\subseteq\cC_{\ds}(X_{n})$.
\item If $\,\cC_{\ds}(X_{n})^{K}\subseteq\cC_{\cusp}(X_{n})^{K}$, then $\cC_{\ds}(X_{n})=\cC_{\cusp}(X_{n})$.
\item Fix $\fa\in\cA$ with $\fa_{\fq}\subseteq \fa$. Let $\phi\in\cC(X_{n})$ be $K$-finite. Then $\phi\in\cC_{\ds}(X_{n})$ if and only if for every $\fh$-compatible minimal parabolic $P\in\cP(\fa)$ and every $k\in K$ the function
    $$
    \Ht_{P}\big(L_k \phi\big)\Big|_{A_{\fq}}
    $$
    is a finite linear combination of exponential functions with non-zero exponents.
\end{enumerate}
\end{Thm}

Theorem \ref{Thm relation cusp forms and discrete series} and the estimates in Theorem \ref{Thm main theorem limit behavior} have the following corollary.

\begin{Thm}\label{Thm C_ds=C_cusp}
$\cC_{\ds}(X_{n})=\cC_{\cusp}(X_{n})$.
\end{Thm}

\begin{proof}
By {\em (i)} and {\em (ii)} in Theorem \ref{Thm relation cusp forms and discrete series}  it suffices to prove that $\cC_{\ds}(X_{n})^{K}\subseteq\cC_{\cusp}(X_{n})^{K}$.
Let $\phi\in\cC_{\ds}(X_{n})^{K}$. We need to show that $\phi$ is a cusp form. For this we will prove that condition {\em(\ref{Prop cusp form def iv})} in Proposition \ref{Prop alternatives for def cusp form} is satisfied.

 By {\em (iii)} in Theorem \ref{Thm relation cusp forms and discrete series} the restriction of $\Ht_{P}\phi$ to $A_{\fq}$ is of exponential type with non-zero exponents. From Theorem \ref{Thm main theorem limit behavior} it follows that this function is bounded. The only function on $A_{\fq}$ that satisfies both conditions is the $0$-function. This proves that
$$
\int_{N_{P}/(N_{P}\cap H)}\phi(an)\,dn=0
$$
for every $a\in A_{\fq}$.
Now, let $g\in G$. By the Iwasawa decomposition there exist $k\in K$, $a\in A_{\fq}$ and $a_{H}\in A\cap H$ such that $g\in kaa_{H}N_{P}$. Using that $\phi$ is $K$-invariant, we find
$$
a_{H}^{\rho_{P}-\rho_{P,\fh}}\int_{N_{P}/(N_{P}\cap H)}\phi(gn)\,dn
=\int_{N_{P}/(N_{P}\cap H)}\phi(an)\,dn
=0.
$$
This proves the claim.
\end{proof}

%%% ----------------------------------------------------------------------
\section{Proof of Theorem \ref{Theorem main theorem convergence} for $\sigma$-parabolic rank $1$}
\label{section Proof of Main theorem 1 for sigma parabolic rank 1}
%%% ----------------------------------------------------------------------
In this section we deal with the proof for Theorem \ref{Theorem main theorem convergence} under the assumption that $P$ is of $\sigma$-parabolic rank $1$. In \ref{subsection Root systems in a} -- \ref{subsection Equivalent theorem} we first reduce the statement to a (seemingly) less general statement, which we then prove in \ref{subsection the integral} -- \ref{subsection Divergence}.

%%% ----------------------------------------------------------------------
\subsection{Root systems.}\label{subsection Root systems in a}
%%% ----------------------------------------------------------------------
We recall the definition of $\fa_{\fq}$, $A_{\fq}$ and $a_{t}$ from Section \ref{subsection the space Xn}, and define
$$
\fa_{\fh}  :=  
\big\{\diag(x_{1},x_{2},\dots,x_{n-1},x_{1}):x_{i}\in\R, 2x_{1}+\sum_{i=2}^{n-1}x_{i}=0\big\}.$$
Note that $\fa_{\fq}$ is a maximal abelian subspace of $\fp\cap\fq$ and $\fa_{\fh}$ is a subspace of $\fp\cap\fh$ such that $\fa=\fa_{\fh}\oplus\fa_{\fq}$ is a maximal abelian subspace of $\fp$. We write $A$ for $\exp(\fa)$

In the remainder of this section we shall describe the root system of
$\fa$ in $\fg.$ 
For $1\leq k\leq n$ we define the functional
$$
e_{k}:\fa\to\R;\qquad \diag(x_{1},\dots,x_{n})\mapsto x_{k}.
$$
The root system of $\fa$ in $\fg$  then equals
$$
\Sigma:=\Sigma(\fg,\fa)
=
\big\{e_{i}-e_{j}:1\leq i,j\leq n, i\neq j\big\}.
$$
The root spaces are $\fg_{e_{i}-e_{j}}=\R E_{i,j}$, where $E_{i,j}$ is the matrix whose entry on the $i^{\textnormal{th}}$ row and $j^{\textnormal{th}}$ column equals $1,$ whereas all remaining entries equal zero.

Note that
$$
\Sigma_{\fh}
:=\Sigma_{\fh}(\fa)
=\big\{e_{i}-e_{j}:2\leq i,j\leq n-1, i\neq j\big\}
$$
is the root system in $\fa_{\fh}$ for $\Cen_{\fh}(\fa_{\fq});$ see Proposition \ref{Prop properties Sigma_h(a)} {\em(ii)}. We set 
\begin{equation}
\label{e: positive system gS fh}
\Sigma_{\fh}^{+}
=\Sigma_{\fh}^{+}(\fa)
=\{e_{i}-e_{j}:2\leq i<j\leq n-1\big\}
\end{equation}
and define
$$
\rho_{\fh}
=\frac{1}{2}\sum_{\alpha\in\Sigma_{\fh}^{+}}\alpha.
$$
The set
$$
\Sigma_{\fq}
=\Sigma\cap\fa_{\fq}^{*}
=\{\pm(e_{1}-e_{n})\}
$$
forms a root system. We denote the associated
positive system $\big\{e_{1}-e_{n}\}$ by $\Sigma_{\fq}^{+}$.

%%% ----------------------------------------------------------------------
\subsection{Classification of parabolic subgroups}
%%% ----------------------------------------------------------------------

Recall that $\cP$ denotes the set of minimal parabolic subgroups containing a $\sigma$-stable maximal connected split abelian subgroup $B,$  such that $\fb\cap \fq$ has maximal dimension, i.e., $\dim(\fb\cap\fq)=1$. Recall also that $\cP(\fa)$ denotes the subset of $\cP$ consisting of minimal parabolic subgroups containing $A.$
 From now on we shall write $\gS(P):= \gS(\fa;P)$ for $P \in \cP(\fa).$ 

\begin{Lemma}\label{lemma for SLn: P is H-conjugate to Q in cP(A) in standard position}
Let $Q\in\cP$. Then there exists a parabolic subgroup $P \in \cP(\fa)$ which
is $H$-conjugate to $Q$ and satisfies
\begin{equation}\label{eq (Sigma_h cup Sigma_q)subseteq Sigma(P)}
(\Sigma_{\fh}^{+}\cup\Sigma_{\fq}^{+})\subseteq\Sigma(P).
\end{equation}
Let $P$ be any such parabolic subgroup. Then $\rho_{P,\fh} = \rho_\fh.$
Moreover, $Q$ is $\fh$-compatible if and only if $P$ is $\fh$-compatible
(see Definition \ref{d: fh compatible}).
\end{Lemma}

\begin{proof}
By Corollaries \ref{Cor H-conjugacy of parabolic subgroups} and \ref{Cor Q in cP(A) is Nor_(K cap H)(fa)-conjugate to ps with Sigma_h^+ subseteq Sigma(P)}, $Q$ is $H$-conjugate to a parabolic subgroup $P'\in\cP(\fa)$ with $\Sigma_{\fh}^{+}\subseteq\Sigma(P').$
If $\Sigma(P')\cap\fa_{\fq}^{*}=\Sigma_{\fq}^{+}$, then
(\ref{eq (Sigma_h cup Sigma_q)subseteq Sigma(P)}) holds with
$P = P'.$
Otherwise, recall the element $k_{0}$ from (\ref{eq def k0}). This element is in $\Nor_{K\cap H}(\fa_{\fq})\cap\Cen_{K\cap H}(\fa_{\fh})$ and acts by inversion on $\fa_{\fq}$. Hence, $P=k_{0}P' k_{0}^{-1}$ satisfies
(\ref{eq (Sigma_h cup Sigma_q)subseteq Sigma(P)}).
In particular, $\rho_{P,\fh} = \rho_\fh.$ The final assertion
follows from Remark \ref{r: h comp for H conjugation}.
\end{proof}

We will now classify the parabolic subgroups $P\in\cP(\fa)$ satisfying (\ref{eq (Sigma_h cup Sigma_q)subseteq Sigma(P)}).
Every parabolic subgroup $P\in\cP(\fa)$ is uniquely determined by the corresponding positive system
$\Sigma(P)$. The set of these positive systems is in bijection with the symmetric group $S_{n}$: an element $\tau\in S_{n}$
corresponds to the positive system $\Sigma(P)$ given by 
$$
\Sigma(P)
=\{e_{\tau^{-1}(i)}-e_{\tau^{-1}(j)}:1\leq i<j\leq n\}.
$$
Equivalently, a root $e_i - e_j \in \Sigma$ belongs to
$\Sigma(P)$ if and only if $\tau (i) < \tau ( j) .$
 We thus see that the parabolic subgroups $P\in\cP(\fa)$ satisfying
(\ref{eq (Sigma_h cup Sigma_q)subseteq Sigma(P)}) correspond to the $\tau\in S_{n}$ such that
\begin{equation}\label{eq standard tau ineq}
\tau(i)<\tau(j)\text{ for }2\leq i<j\leq n-1\quad\text{and}\quad\tau(1)<\tau(n).
\end{equation}
Given such a permutation $\tau\in S_{n}$, there exists a unique $k$, with $2\leq k\leq n$, such that
\begin{equation}\label{eq tau k ineq}
\tau(k-1)\leq \tau(1)<\tau(k),
\end{equation}
and a unique $l$, with $k\leq l\leq n$, such that
\begin{equation}\label{eq tau l ineq}
\tau(l-1)< \tau(n)\leq \tau(l).
\end{equation}
Conversely, for each pair of integers $(k,l)$ satisfying $2\leq k\leq l\leq n$, there exists a unique $\tau\in S_{n}$ satisfying (\ref{eq standard tau ineq}), (\ref{eq tau k ineq}) and (\ref{eq tau l ineq}). From now on we write $P_{k,l}$ for the corresponding minimal parabolic subgroup and $\Sigma^{+}_{k,l}$ for $\Sigma(P_{k,l})$. Furthermore, we write $P_{k,l}=MAN_{k,l}$ for the Langlands decomposition of $P_{k,l}$.
For future reference we note that the positive system corresponding to $(k,l)$ is
given by the disjoint union
\begin{align}\label{eq Sigma_k,l}
\Sigma^{+}_{k,l}
=\Sigma_{\fh}^{+}\cup\Sigma_{\fq}^{+}
&\cup\big\{e_{j}-e_{1}:2\leq j\leq k-1\big\}
\cup\big\{e_{1}-e_{j}:k\leq j\leq n-1\big\}\\
\nonumber\qquad&\cup\big\{e_{j}-e_{n}:2\leq j\leq l-1\big\}
\cup\big\{e_{n}-e_{j}:l\leq j\leq n-1\big\}.
\end{align}

In the following we assume that $b$ is the $\ad$-invariant
bilinear form on $\fsl(n, \R)$ given by $b(X,Y) = \tr(XY).$
It is well known that $b$ is a positive multiple of the Killing form $B,$
in fact, $b = \frac{1}{4n} B.$ The restriction of $b$ to $\fa,$
denoted $\inp{\dotvar}{\dotvar},$ is a positive
definite inner product. We equip $\fad$ with the dual inner product,
also denoted $\inp{\dotvar}{\dotvar},$ which on two elements
$\mu,\nu \in \fad$ is given as follows. The elements have unique
expressions $\mu = \sum_j \mu_j e_j$ and $\nu = \sum_j \nu_j  e_j$
provided we demand that  $\sum_j \mu_j =0 $ and $ \sum_j \nu_j = 0.$
In terms of these expressions,
\begin{equation}
\label{e: inner product for roots on a}
\inp{\mu}{\nu} = \sum_j \mu_j\nu_j.
\end{equation}
In particular, each root $e_i - e_j$ has length $\sqrt 2.$

\begin{Lemma}\label{Lemma langle e_i-e_j,rho_h rangle}
Let $i$ and $j$ be integers such that $1\leq i<j\leq n$. Then
$$
\langle e_{i}-e_{j},\rho_{\fh}\rangle=
\begin{cases}
j-i & \text{if }2\leq i<j\leq n-1,\\
j-\frac{n+1}{2} & \text{if }i=1\text{ and }2\leq j\leq n-1,\\
\frac{n+1}{2}-i & \text{if }2\leq i\leq n-1\text{ and }j=n,\\
0 & \text{if }i=1\text{ and }j=n.
\end{cases}
$$
In particular, $\langle\alpha,\rho_{\fh}\rangle\geq 0$ for every
$\alpha\in\Sigma(P_{k,l})$ if and only if $\frac{n+1}{2}\leq k\leq l\leq \frac{n+3}{2}$.
\end{Lemma}

\begin{proof}
Using the definition of $\rho_{\fh}$, we find
\begin{align*}
2\rho_{\fh}
&=\sum_{2\leq i<j\leq n-1}(e_{i}-e_{j})
=\sum_{i=2}^{n-1}(n+1-2i)e_{i}.
\end{align*}
The first statement in the lemma is a direct consequence of this formula; the second follows from comparison with (\ref{eq Sigma_k,l}).
\end{proof}

\begin{Lemma}\label{Lemma n_k,l cap h}
Let $2\leq k\leq l\leq n$ and let $\fn_{k,l}$ be the Lie algebra of $N_{k,l}$. Then
\begin{equation}\label{eq n_k,l cap h}
\fn_{k,l}\cap\fh=
\bigoplus_{\alpha\in\Sigma_{\fh}^{+}}\fg_{\alpha}
    \oplus\bigoplus_{2\leq j\leq k-1}(1+\sigma)\fg_{e_{j}-e_{1}}
    \oplus\bigoplus_{l\leq j\leq n-1}(1+\sigma)\fg_{e_{n}-e_{j}}.
\end{equation}
\end{Lemma}

\begin{proof}
Since $\sigma(e_{1})=e_{n}$ and $\sigma(e_{j})=e_{j}$ for $2\leq j\leq n-1$, we find
\begin{align*}
\Sigma^{+}_{k,l}\cap\sigma(\Sigma^{+}_{k,l})
=\Sigma_{\fh}^{+}
&\cup\big\{e_{j}-e_{1}:2\leq j\leq k-1\big\}
\cup\big\{e_{1}-e_{j}:l\leq j\leq n-1\big\}\\
\qquad&\cup\big\{e_{j}-e_{n}:2\leq j\leq k-1\big\}
\cup\big\{e_{n}-e_{j}:l\leq j\leq n-1\big\},
\end{align*}
the union being disjoint.
The root spaces $\fg_{\alpha}$ for $\alpha\in \Sigma^{+}_{\fh}$ are contained in $\fh$. Furthermore, if both $\alpha$ and $\sigma(\alpha)$ are roots in $\Sigma^{+}_{k,l}\setminus\Sigma^{+}_{\fh}$, then
$$
\fg_{\alpha}\oplus\fg_{\sigma(\alpha)}
=(1+\sigma)\fg_{\alpha}\oplus(1-\sigma)\fg_{\alpha},
$$
where the first term in the right-hand side is a subspace of $\fn_{k,l}\cap\fh$ and the second term is a subspace in $\fn_{k,l}\cap\fq$.
This proves that the right-hand side of (\ref{eq n_k,l cap h}) is contained in $\fn_{k,l}\cap \fh$.

To prove the converse, assume that $Y\in\fn_{k,l}\cap \fh$. Then $Y=\sum_{\alpha\in\Sigma_{k,l}^{+}}Y_{\alpha}$, where $Y_{\alpha}\in\fg_{\alpha}$. Since $Y\in\fh$, we have $\sigma(Y)=Y$. This implies that $\sigma(Y_{\alpha})=Y_{\sigma(\alpha)}$ if both $\alpha$ and $\sigma(\alpha)$ are elements of $\Sigma_{k,l}^{+}$ and $Y_{\alpha}=0$ otherwise. Let $\alpha$ be a root such that $Y_{\alpha}\neq 0$. If $\alpha=\sigma(\alpha)$ then $\alpha\in\Sigma_{\fh}^{+}$. If $\alpha\neq\sigma(\alpha)$, then there exist a $2\leq j\leq k-1$ such that either $\alpha=e_{j}-e_{1}$ or $\sigma(\alpha)=e_{j}-e_{1}$, or there exists a $l\leq j\leq n-1$ such that either $\alpha=e_{n}-e_{j}$ or $\sigma(\alpha)=e_{n}-e_{j}$. Therefore, $Y$ is contained in the right-hand side of (\ref{eq n_k,l cap h}).
This proves the lemma.
\end{proof}

\begin{Prop}\label{Prop P H-conj to P_k,l; H-comp iff (n+1)/2 leq k,l leq (n+3)/2}
Let $P\in\cP$. Then there exist unique integers $k$ and $l$, with $2\leq k\leq l\leq n$, such that $P$ is $H$-conjugate to $P_{k,l}$. Moreover, $P\in\cPH$ if and only if
$$\frac{n+1}{2}\leq k\leq l\leq \frac{n+3}{2}$$
and $P\in\cP_{*}$ if and only if $k=l$.
\end{Prop}

\begin{proof}
The existence of the integers $k$ and $l$ follows directly from Lemma \ref{lemma for SLn: P is H-conjugate to Q in cP(A) in standard position}. To prove uniqueness of these integers, let $2\leq k\leq l\leq n$ and $2\leq k'\leq l'\leq n,$ and assume that there exists an $h\in H$ with $hP_{k,l}h^{-1}=P_{k',l'}$. Note that $A$ and $hAh^{-1}$ are both $\sigma$-stable split components of $P_{k',l'}$. From Theorem \ref{Thm P contains sigma-stable A, unique up to N cap H conjugation} it follows that there exists a unique $n\in N_{k',l'}\cap H$ such that $nhAh^{-1}n^{-1}=A$. We write $h'$ for $nh$. Note that $h'\in\Nor_{H}(\fa)=\Nor_{H}(\fa_{\fq})\cap\Nor_{H}(\fa_{\fh})$ and $h'P_{k,l}h'^{-1}=P_{k',l'}$. It follows that $h'$ acts trivially on $\Sigma_{\fq}^{+}$ and therefore $h'\in\Cen_{H}(\fa_{\fq})$. From Corollary \ref{Prop properties Sigma_h(a)} we see that $\Sigma_{\fh}$ is the root system of $\fa_{\fh}$ in $\Cen_{\fh}(\fa_{\fq}).$ Since $h'\in\Nor_{H}(\fa_{\fh})\cap\Cen_{H}(\fa_{\fq})$, it induces an element $w$ in the Weyl group of this root system. As the positive system $\Sigma_{\fh}^{+}$ is contained in both $\Sigma_{k,l}^{+}$ and $\Sigma_{k',l'}^{+}$, it follows that $w$ acts trivially on $\Sigma_{\fh}^{+}$, and thus we conclude that $h'$ acts trivially on $\fa_{\fh}$. This proves that $h'P_{k,l}h'^{-1}=P_{k,l}$ and hence $k=k'$ and $l=l'$.

From Lemma \ref{Lemma langle e_i-e_j,rho_h rangle} it is easily seen that $P\in\cPH$ if and only if
$\frac{n+1}{2}\leq k\leq l\leq \frac{n+3}{2}$. The final claim, that $P\in\cP_{*}$ if and only if $k=l,$ follows from Lemma \ref{Lemma n_k,l cap h}.
\end{proof}

\begin{Rem}\label{r: calculation rho zero}
We recall from \cite[Def.~1.1]{vdBanKuit_HC-TransformAndCuspForms} that a parabolic subgroup $P \in \cP(\fa)$ is said
to be $\fq$-extreme if $ \gS(P) \cap \gs\Cartan  \gS(P) = \gS(P)\setminus \gS_\fh.$
Since $\gs\Cartan e_j =  - e_j$ for $2\leq j \leq n-1,$ whereas $\gs \Cartan e_1 = - e_n,$
it follows from the characterization of $\gS_{k,l}$ above that $P_{k,l}$
is $\fq$-extreme if and only if $k=2$ and $l = n-1.$ Let $\gS_0^+$ be the positive
system for the root system $\Sigma(\fg, \faq)$ obtained by restricting the roots
from $\gS_{2, n-1}\setminus \gS_\fh$ to $\faq.$ Then $\gS_0^+ = \{\ga, 2 \ga\},$
where $2\ga = e_1 -e_n.$ It is now readily checked that $\ga$ has multiplicity
$2(n-2),$ whereas $2 \ga$ has multiplicity $1.$ Accordingly, $\rho_0:= \rho(\gS_0^+)$
is given by $\rho_0 = (n - 1)\ga,$ so that
$$
a_t^{\rho_0} = e^{(n-1) t} \qquad (t \in \R).
$$
The element $\rho_0 \in \faqd$ defined above corresponds to the element
$\rho_Q$ appearing in \cite[Prop.~17.2]{vdBan_PrincipalSeriesII}.
Accordingly, it follows that Definition \ref{defi cC} of the Schwartz space $\cC(X_n)$
is consistent with the characterisation given in \cite[Thm.~17.1]{vdBan_PrincipalSeriesII}.
\end{Rem}

%%% ----------------------------------------------------------------------
\subsection{Decomposition and invariant measures}
%%% ----------------------------------------------------------------------
Let $k$ and $l$ be integers such that $2\leq k\leq l\leq n$.
We recall that $P_{k,l}=\Cen_{K}(\fa)AN_{k,l}$ is the minimal parabolic subgroup containing $A$ such that $\Sigma(P_{k,l})=\Sigma_{k,l}^{+}$, and that the latter root system is given by (\ref{eq Sigma_k,l}).

We define
$$
\ns
=\fg_{e_{1}-e_{n}}
    \oplus\bigoplus_{2\leq j\leq l-1}\fg_{e_{j}-e_{n}}
    \oplus\bigoplus_{k\leq j\leq n-1}\fg_{e_{1}-e_{j}}.
$$
Note that $\ns$ is a Lie subalgebra of $\fn_{k,l}$. We write $\Ns$ for the Lie subgroup $\exp(\ns)$.
For $x,y\in\R^{n-2}$ and $z\in\R$, we define
\begin{equation}\label{eq def u_(x,y,z)}
u_{x,y,z}=\left(
   \begin{array}{c:c:c}
    1   &   x^{t}   &   z   \\
   \hdashline
        &   I_{n-2} &   y   \\
   \hdashline
        &           &   1   \\
   \end{array}
 \right).
\end{equation}
A straightforward computation shows that
$$
\Ns
=\big\{u_{x,y,z}
    :x\in\{0\}^{k-2}\times\R^{n-k},y\in\R^{l-2}\times\{0\}^{n-l}, z\in\R\big\}.
$$

By Lemma \ref{Lemma n_k,l cap h}, we have
\begin{equation}\label{eq n=(n cap h)+u}
\fn_{k,l}=(\fn_{k,l}\cap\fh)\oplus\ns.
\end{equation}

\begin{Lemma}\label{Lemma decomposition of Haar measures}
The map
$$\Ns\times(N_{k,l}\cap H)\to N_{k,l};\qquad(u,n)\mapsto un$$
is a diffeomorphism. There exist normalizations for the invariant measure $dx$ on $N_{k,l}/N_{k,l}\cap H$ and the Haar measure $du$ of $\Ns$ such that for every $\phi\in C_c\big(N_{k,l}/(N_{k,l}\cap H)\big)$
\begin{align}
\label{e: integral over U k l}
\int_{N_{k,l}/(N_{k,l}\cap H)}\phi(x)\,dx
&=\int_{\Ns}\phi\big(u\cdot(N_{k,l}\cap H)\big)\,du.
\end{align}
Finally, the normalizations may be chosen such that, in addition,
the above integrals equal
\begin{equation}
\label{e: second integral over U k l}
\int_{z\in\R}\int_{y\in\R^{l-2}\times\{0\}^{n-l}}\int_{x\in\{0\}^{k-2}\times\R^{n-k}}
    \phi\big(u_{x,y,z}\cdot (N_{k,l} \cap H)\big)\,dx\,dy\,dz.
\end{equation}
\end{Lemma}

\begin{proof}
The first two assertions follow from  \cite[Prop.~2.16]{BalibanuVdBan_ConvexityTheoremsForSemisimpleSymmetricSpaces}.
For the final assertion, we note that
$$
u_{x,y, z + \inp{x}{y}/2} = \exp
\left(
   \begin{array}{c:c:c}
    0   &   x^{t}   &   z   \\
   \hdashline
        &   0_{n-2} &   y   \\
   \hdashline
        &           &   0   \\
   \end{array}
 \right).
 $$
 It follows that, up to suitable normalization of measures, the second
 integral in (\ref{e: integral over U k l}) equals
 $$
 \int_{z\in\R}\int_{y\in\R^{l-2}\times\{0\}^{n-l}}\int_{x\in\{0\}^{k-2}\times\R^{n-k}}
    \phi\big(u_{x,y,z + \inp{x}{y}/2}\cdot (N_{k,l} \cap H)\big)\,dx\,dy\,dz.
 $$
 The equality with (\ref{e: second integral over U k l}) now follows
 from a simple substitution of variables.
\end{proof}

To conclude this section we state one more lemma.

\begin{Lemma}\label{Lemma equivalences between parabolics}
Let $w$ be the longest Weyl group element for $\Sigma_{\fh}=\Sigma\big(\Cen_{\fh}(\fa_{\fq});\fa_{\fh}\big),$
relative to the positive system (\ref{e: positive system gS fh}),
and let $w_{0} \in N_{K\cap H}(\fa) \cap \Cen_{K\cap H}(\fa_{\fq})$ be a representative for $w$. Then
$$
\sigma\theta(w_{0}P_{k,l}w_{0}^{-1})=P_{n+2-l,n+2-k}.
$$
Moreover, if $\phi\in C_{c}^{\infty}(X_{n})$, then
$$
\int_{\Ns[n+2-l,n+2-k]}\phi(u\cdot H)\,du
=\int_{\Ns}\phi\Big(\sigma\theta (w_{0}u)\cdot H\Big)\,du.
$$
\end{Lemma}

\begin{proof}
From the identities
$$
w\cdot e_{i}
=\begin{cases}
    e_{1} & i=1\\
    e_{n+1-i} & 2\leq i\leq n-1\\
    e_{n} & i=n
\end{cases}
\quad\text{and}\quad
\sigma\theta e_{i}
=\begin{cases}
    -e_{n} & i=1\\
    -e_{i} & 2\leq i\leq n-1\\
    -e_{1} & i=n,
\end{cases}
$$
we obtain that $\sigma\theta(w\cdot \Sigma^{+}_{k,l})=\Sigma^{+}_{n+2-l,n+2-k}$. This proves the first statement. The second statement follows from the first as
$$
\Ns\to\Ns[n+2-l,n+2-k];\qquad u\mapsto\sigma\theta(w_{0}uw_{0}^{-1})
$$
is a diffeomorphism with Jacobian $1$.
\end{proof}

%%% ----------------------------------------------------------------------
\subsection{Reduction to an equivalent theorem}\label{subsection Equivalent theorem}
%%% ----------------------------------------------------------------------
The space $\cC(X_{n})$ is stable under translation by $G$ and
pull-back by $\sigma \theta$. Under the assumption that $P$ is of $\sigma$-parabolic rank $1$, we conclude from Proposition \ref{Prop P H-conj to P_k,l; H-comp iff (n+1)/2 leq k,l leq (n+3)/2}, Lemma \ref{Lemma decomposition of Haar measures} and Lemma \ref{Lemma equivalences between parabolics} that Theorem \ref{Theorem main theorem convergence} is equivalent to the following theorem.

\begin{Thm}\label{Thm alternative for main theorem}
Let $k$ and $l$ be integers such that $2\leq k\leq l \leq n+2-k$. Then the integral
$$
\int_{\Ns}\phi(u\cdot H)\,du
$$
is absolutely convergent for every $\phi\in\cC(X_{n})$ if and only if
$\frac{n+1}{2}\leq k\leq l\leq\frac{n+3}{2}$.
\end{Thm}

From now on we assume that $k$ and $l$ are integers such that $2\leq k\leq l \leq n+2-k$.
%%% ----------------------------------------------------------------------
\subsection{The integral}\label{subsection the integral}
%%% ----------------------------------------------------------------------

\begin{Lemma}\label{Lemma 2cosh(4t)=P(x,y,z)}
Let $x,y\in\R^{n-2}$ and $z\in\R$. If $u_{x,y,z}\cdot H\in Ka_{t}\cdot H$, then
\begin{align}
2\cosh(4t)
=&(1-z+\langle x,y\rangle)^{2}(1+z)^{2}
+(1-z+\langle x,y\rangle)^{2}\|y\|^{2}
+(\langle x,y\rangle-z)^{2}\nonumber
\\
\label{eq 2 cosh(4t)}
&+(1+z)^{2}\|x\|^{2}
+2\langle x,y\rangle+\|x\|^{2}\|y\|^{2}
+\|x\|^{2}
+z^{2}
+\|y\|^{2}
+1.
\end{align}
We denote the orthogonal projection $\R^{n-2}\to\{0\}^{k-2}\times \R^{l-k}\times\{0\}^{n-l}$
by $\pi$. Let $x\in\{0\}^{k-2}\times\R^{n-k}$, $y\in\R^{l-2}\times\{0\}^{n-l}$ and $z\in\R$.
If $u_{x,y,z}\cdot H\in Ka_{t}\cdot H$, then
\begin{equation}\label{2nd eq cosh}
2\cosh(4t)
=\langle x,A_{y,z}x\rangle + \langle b_{y,z},x\rangle+c_{y,z},
\end{equation}
where
\begin{align*}
A_{y,z}
&=\big(\|y\|^{2}+(1+z)^{2}+1\big)\Big(I_{n-2}+\pi(y)\pi(y)^{t}\Big)\in\Aut(\R^{n-2}),\\
b_{y,z}
&=2(1-z)\big(\|y\|^{2}+(1+z)^{2}+1\big)\pi(y)\in\{0\}^{k-2}\times\R^{l-k}\times\{0\}^{n-l},\\
c_{y,z}
&=(1-z)^{2}\big(\|y\|^{2}+(1+z)^{2}+1\big)+(z^{2}+2z+\|y\|^{2})\in\R_{\geq2}.
\end{align*}
\end{Lemma}

\begin{proof}
Straight-forward computations show that
$$
u_{x,y,z}\sigma(u_{x,y,z})^{-1}
=u_{x,y,z}S u_{x,y,z}^{-1}S^{-1}
=\left(
   \begin{array}{c:c:c}
     (\langle x,y\rangle-z+1)(1+z)    &   (1+z)x^{t}      &   z\\
     \hdashline
     (\langle x,y\rangle-z+1)y     &   I_{n-2}+yx^{t}  &   y\\
     \hdashline
     \langle x,y\rangle-z         &x^{t}              &   1\\
   \end{array}
 \right)
$$
and
$$
\|I_{n-2}+yx^{t}\|_{HS}^{2}
=n-2+2\langle x,y\rangle+\|x\|^{2}\|y\|^{2}.
$$
Equation (\ref{eq 2 cosh(4t)}) now follows from Lemma \ref{Lemma KAH decomposition}.
Equation (\ref{2nd eq cosh}) is a direct consequence of (\ref{eq 2 cosh(4t)}).
\end{proof}

\begin{Cor}\label{Cor integral equality}
Let $\phi\in C(X_{n})$ be $K$-invariant and non-negative. Let
$\bar{\phi}: \R_{\geq 2} \to \R_{\geq 0} $  be given by
$$
\bar{\phi}(2\cosh4t)=\phi(a_{t}\cdot H).
$$
(Note that the function $\R\ni t\mapsto \phi(a_{t}\cdot H)$ is even since $\phi$ is $K$-invariant.)
Furthermore, let
$J_{k}: \R^{n-2}\times\R\to\R_{>0}$
 be given by
$$
J_{k}(y,z)=\big(\|y\|^{2}+(1+z)^{2}+1\big)^{\frac{k-n}{2}}(1+\|\pi(y)\|^{2})^{-\frac{1}{2}}
$$
and let $c':\R^{n-2}\times\R\to\R_{\ge2}$ be given by
$$
c'(y,z)=(1-z)^{2}\Big(\|y\|^{2}+(1+z)^{2}+1\Big)\frac{1}{1+\|\pi(y)\|^{2}}\,+ \,\big(z^{2}+2z+\|y\|^{2}\big).
$$
Then
$$
\int_{\Ns}\phi(u \cdot H)\,du
=\int_{z\in\R}\int_{y\in\R^{l-2}\times\{0\}^{n-l}}
    J_{k}(y,z)\int_{x\in\{0\}^{k-2}\times\R^{n-k}}
    \bar{\phi}\big(\|x\|^{2}+c'(y,z)\big)\,dx\,dy\,dz.
$$
\end{Cor}

\begin{proof}
First, note that, by Lemma \ref{Lemma decomposition of Haar measures},
$$
\int_{\Ns}\phi(u\cdot H)\,du
=\int_{z\in\R}\int_{y\in\R^{l-2}\times\{0\}^{n-l}}\int_{x\in\{0\}^{k-2}\times\R^{n-k}}
    \phi(u_{x,y,z}\cdot H)\,dx\,dy\,dz.
$$
We will use Lemma \ref{Lemma 2cosh(4t)=P(x,y,z)} to
rewrite this integral.
Note that the restriction $B_{y,z}$ of $A_{y,z}$ to $\{0\}^{k-2}\times\R^{n-k}$ is a positive definite
symmetric automorphism of $\{0\}^{k-2}\times\R^{n-k}$. We define
$B_{y,z}^{\scriptscriptstyle{1/2}}$
and $B_{y,z}^{\scriptscriptstyle{-1/2}}$ 
 to be the square root (defined in the usual way) and the inverse
square root of $B_{y,z}$, respectively. Now we
apply the substitution of variables
$x'=B_{y,z}^{\scriptscriptstyle{1/2}}x+\frac{1}{2}B_{y,z}^{\scriptscriptstyle{-1/2}}b_{y,z}$ to the inner integral.
When $u_{x,y,z}\cdot H\in Ka_t\cdot H$ we obtain from (\ref{2nd eq cosh}) that 
$$
2\cosh(4t)=\|x'\|^{2}+c'_{y,z},
$$
where
$$
2\le c'_{y,z}=c_{y,z}-\frac{1}{4}\|B_{y,z}^{-\frac{1}{2}}b_{y,z}\|^{2}.
$$
The Jacobian of the substitution equals the determinant of $B_{y,z}^{\scriptscriptstyle{-1/2}}$. In turn, this determinant equals $(\det B_{y,z})^{\scriptscriptstyle{-1/2}}=J_{k}(y,z)$.
Observing that $b_{y,z}$ is an eigenvector of $B_{y,z}$ with eigenvalue
$\big(\|y\|^{2}+(1+z)^{2}+1\big)(1+\|\pi(y)\|^{2})$, we see that
$$
\frac{1}{4}\|B_{y,z}^{-\frac{1}{2}}b_{y,z}\|^{2}
=(1-z)^{2}\big(\|y\|^{2}+(1+z)^{2}+1\big)\frac{\|\pi(y)\|^{2}}{1+\|\pi(y)\|^{2}}.
$$
Hence $c'_{y,z}=c'(y,z)$, and the corollary is proved.
\end{proof}

%%% ----------------------------------------------------------------------
\subsection{The case of convergence}
%%% ----------------------------------------------------------------------

In view of Remark \ref{r: Schwartz seminorms}, the following proposition implies the `if' part of Theorem
\ref{Thm alternative for main theorem}.

\begin{Prop}\label{Prop int phi convergent}
Assume $\frac{n+1}2\le k\le l\le \frac{n+3}2$.
For  $m \in \R,$  let $\phi_m:$ $X_{n}\to\R$ be defined by
$$
\phi_m(k'a_{t}\cdot H):= (2\cosh 4t)^{\frac{1-n}4}\big(1+\log(2\cosh(4t))\big)^{-m}
\qquad (k'\in K, t\in\R).
$$
Then there exists $m\ge 0$ such that the integral
$\displaystyle\int_{\Ns}\phi_m(u\cdot H)\,du$
is convergent.
\end{Prop}

\begin{proof}
Let $f_{m}(u)=u^{\frac{1-n}4}(1+\log u)^{-m}$ for $u\ge 1$.
Then $f $ equals $\bar \phi_m$ as defined in Corollary \ref{Cor integral equality},
and we see that it suffices to show that
$$
I_m:=\int_{z\in\R}\int_{y\in\R^{l-2}\times\{0\}^{n-l}}
    J_{k}(y,z)\int_{x\in\{0\}^{k-2}\times\R^{n-k}}
    f_{m}\big(\|x\|^{2}+c'(y,z)\big)\,dx\,dy\,dz <\infty,
$$
for $m\ge 0$ sufficiently large. We substitute $x=c'(y,z)^{1/2}\xi$
in the inner integral and obtain that $I_{m}$ is equal to
$$
\int_{z\in\R}\int_{y\in\R^{l-2}\times\{0\}^{n-l}}
    J_{k}(y,z)c'(y,z)^{\frac{n-k}2}\int_{\xi\in\{0\}^{k-2}\times\R^{n-k}}
    f_{m}\big(c'(y,z)(\|\xi\|^{2}+1)\big)\,d\xi\,dy\,dz.
$$
Observe that if $m=m_1+m_2$ with $m_1,m_2\ge 0$, then
$$ \big(1+\log\big[c'(y,z)(\|\xi\|^2+1)\big]\big)^{-m}\le
\big(1+\log c'(y,z)\big)^{-m_2}\big(1+\log(\|\xi\|^2+1)\big)^{-m_1}.$$
Hence,
\begin{align}
I_m\le \int_{z\in\R}&\int_{y\in\R^{l-2}\times\{0\}^{n-l}}
    J_{k}(y,z)c'(y,z)^{\frac{n+1}4-\frac k2}
\big(1+\log c'(y,z)\big)^{-m_2}\,dy\,dz \nonumber \\
&\times\int_{\xi\in\{0\}^{k-2}\times\R^{n-k}}
    (\|\xi\|^{2}+1)^{\frac{1-n}4}\big(1+\log(\|\xi\|^2+1)\big)^{-m_1}\,d\xi.
    \label{e: long integral estimating I m}
\end{align}
Since 
$k\ge \frac{n+1}2,$ we have $\frac{n-1}2\ge n-k$ and see that
the integral over $\xi$ converges for $m_1 \geq 2$
(use polar coordinates).

For the integral over $(y,z)$
we shall need the following estimate of $c'(y,z)$.
We write $y=(v,w,0),$ where
$v\in\R^{k-2}$ and $w=\pi(y)\in\R^{l-k},$
and claim that
\begin{equation}\label{estimate c'}
c'(y,z)\ge \tfrac14(\|w\|^2+1)^{-1}
\big((1-z)^2+1\big)\big(\|v\|^2+(1+z)^2+1\big)+\|w\|^2+1.
\end{equation}
To verify the claim, we note that 
\begin{align*}
c'(y,z)&=(\|w\|^2+1)^{-1}(1-z)^2\big(\|v\|^2+\|w\|^2+(1+z)^2+1\big)
+z^2+2z+\|v\|^2+\|w\|^2\\
&=(\|w\|^2+1)^{-1}
\Big(\big((1-z)^2+1+\|w\|^2\big)\|v\|^2
+z^4+1+2z^2\|w\|^2\Big)+\|w\|^2+1\\
&\ge (\|w\|^2+1)^{-1}
\Big(\big((1-z)^2+1\big)\|v\|^2
+z^4+1\Big)+\|w\|^2+1\\
&\ge \tfrac14(\|w\|^2+1)^{-1}
\Big(\big((1-z)^2+1\big)\|v\|^2
+z^4+4\Big)+\|w\|^2+1.
\end{align*}
Using
$$
\big((1-z)^2+1\big)\big((1+z)^2+1\big)=z^4+4,
$$
we obtain the validity of claim (\ref{estimate c'}).

We first assume that $k=l$. Then $w=0$ and we obtain
\begin{equation*}
c'(y,z)\ge
\tfrac14 \big((1-z)^2+1\big)\big(\|v\|^2 +(1+z)^2+1\big)+1
\end{equation*}
and
$$
J_k(y,z) = (\|v\|^2 + (1 + z)^2 +1)^{\frac{k-n}{2}}.
$$ 
Hence, there exists a constant $C>0$, independent of $m$, such that
\begin{align*}
I_m
&\le C\int_{z\in\R}\int_{v\in\R^{k-2}}
\big(\|v\|^2+(1+z)^2+1\big)^{\frac{1-n}4}
\big((1-z)^2+1\big)^{\frac{n+1}4-\frac k2}\\
&\qquad\qquad\times\Big(1+\log\big[1+\tfrac14\big((1-z)^2+1\big)\big(\|v\|^2+(1+z)^2+1\big)\big]
\Big)^{-m_2}\,dv\,dz.
\end{align*}
The substitution of $v=((1+z)^2+1)^{1/2}\eta$ now allows us to
estimate $I_m$  by a constant times the product of the integrals
\begin{equation}
\label{e: integral as above}
\int_{\eta\in\R^{k-2}}
(\|\eta\|^2+1)^{\frac{1-n}4}
\big(1+\log(\|\eta\|^2+1)\big)^{-m_3}\,dy
\end{equation}
and
$$
\int_{z\in\R}
\big((1+z)^2+1\big)^{\frac k2-\frac{n+3}4}
\big((1-z)^2+1\big)^{\frac{n+1}4-\frac k2}
\big(1+\log(\tfrac14z^4+2)\big)^{-m_4}
\,dz
$$
where $m_2=m_3+m_4$. In the above we already saw that
(\ref{e: integral as above}) converges for $m_3 \geq 2.$
The remaining integral is easily seen to converge for $m_4 \geq 2.$ 

Next we assume $k<l$. Then $n$ is odd and $k=\frac{n+1}2$,
$l-k=1$. Hence,
the power of $c'(y,z)$ in the first integral in
(\ref{e: long integral estimating I m}) equals zero. We thus see that 
$I_m$ is bounded by a constant times
$$
\int_{z\in\R}\int_{v\in\R^{k-2}}\int_{w\in\R}
\big(\|v\|^{2}+w^2+(1+z)^{2}+1\big)^{\frac{1-n}{4}}(w^{2}+1)^{-\frac{1}{2}}
\big(1+\log c'(y,z)\big)^{-m_2}\,dw\,dv\,dz.
$$
Furthermore,
\begin{equation*}
c'(y,z)\ge
\tfrac14(w^2+1)^{-1}\big(\|v\|^2+(1+z)^2+1\big)+w^2+1.
\end{equation*}
The substitutions $v=(w^2+1)^{1/2}\eta$
and $1+z=(w^2+1)^{1/2}\zeta$ then
allow us to
estimate by the product of
$$
\int_{(\eta,\zeta)\in\R^{k-2}\times\R}
(\|\eta\|^2+\zeta^2+1)^{\frac{1-n}4}
\big(1+\log(\|\eta\|^2+\zeta^2+1)\big)^{-m_3}\,d(\eta,\zeta)
$$
and
$$
\int_{w\in\R}
(w^2+1)^{-\frac12}
\big(1+\log(w^2+1)\big)^{-m_4}
\,dw,
$$
where $m_2=m_3+m_4$.
Both integrals are easily seen to converge for $m_3,m_4\geq 2.$
\end{proof}

%%% ----------------------------------------------------------------------
\subsection{The case of divergence}\label{subsection Divergence}
%%% ----------------------------------------------------------------------

The following proposition and its corollary imply the `only if'
part of Theorem
\ref{Thm alternative for main theorem}.

\begin{Prop}\label{Prop int phi divergent}
Let $\nu\in\R$ and let $\phi_{\nu}$ be the function $X_{n}\to\R$ given by
$$
\phi_{\nu}(k'a_{t}\cdot H)=  (2 \cosh 4t)^{\nu}
\qquad (k'\in K, t\in\R).
$$
Then the integral
$\displaystyle\int_{\Ns}\phi_{\nu}(u\cdot H)\,du$
is divergent for $\nu>\min\{\frac{k-n}{2},\frac{2-l}{2}\}$.
\end{Prop}

\begin{proof}
Let $\bar \phi_\nu$ be the function associated to $\phi_\nu$ as in
Corollary \ref{Cor integral equality}. Then $\bar\phi_\nu(r) = r^\nu$ for
$r \geq 1.$ Clearly, the integral
$$
\int_{x\in\R^{n-k}}\big(\|x\|^{2}+c\big)^{\nu}\,dx
$$
is divergent for all positive constants $c$ if $\nu>\frac{k-n}{2}.$  Hence, it follows
from Corollary \ref{Cor integral equality} that
 $\int_{\Ns}\phi_{\nu}(u\cdot H)\,du$ is divergent for such $\nu$.
 Note that $\phi_{\nu}$ satisfies
$$
\phi_{\nu}\Big( \sigma \theta (kg)\cdot H\Big)
=\phi_{\nu}(g\cdot H)\qquad(g\in G, k\in K).
$$
Combining this with Lemma \ref{Lemma equivalences between parabolics}, we infer that
$$
\int_{\Ns[n+2-l,n+2-k]}\phi_{\nu}(u\cdot H)\,du
=\int_{\Ns}\phi_{\nu}(u\cdot H)\,du.
$$
By the previous argument, we obtain that the first, hence also the second integral is
divergent for $\nu>\frac{(n+2-l)-n}{2}=\frac{2-l}{2}$. The assertion now follows.
\end{proof}

\begin{Cor}\label{Cor int_(U) divergent}
Assume $k<\frac{n+1}{2}$ or $l>\frac{n+3}{2}$. Then there exists a function $\phi\in\cC(X_{n})$ such that
$\displaystyle{\int_{\Ns}\phi(u\cdot H)\,du}$ is divergent.
\end{Cor}

\begin{proof}
Assume $k<\frac{n+1}{2}$ or $l>\frac{n+3}{2}$. Then $\min\{\frac{k-n}{2},\frac{2-l}{2}\}<\frac{1-n}{4}$. We may therefore take $\nu$ such that $\min\{\frac{k-n}{2},\frac{2-l}{2}\}<\nu<\frac{1-n}{4}$. Then $\phi_{\nu}\in\cC(X_{n})$ by Lemma \ref{Lemma phi_nu is Schwartz}. The result now follows by application of Proposition \ref{Prop int phi divergent}.
\end{proof}

%%% ----------------------------------------------------------------------
\section{Proof of Theorem \ref{Theorem main theorem convergence} for $\sigma$-parabolic rank $0$}
\label{section Proof of Main theorem 1 for sigma parabolic rank 0}
%%% ----------------------------------------------------------------------
We now turn to the proof of Theorem \ref{Theorem main theorem convergence} under the assumption that $P$ is of $\sigma$-parabolic rank $0$. In \ref{subsection root systems in b} -- \ref{subsection equivalent theorem for b} we first reduce the statement to a (seemingly) less general statement, which we then prove in \ref{subsection integral for b} -- \ref{subsection divergence for b}.

%%% ----------------------------------------------------------------------
\subsection{Root systems}\label{subsection root systems in b}
%%% ----------------------------------------------------------------------
Recall the element $\kappa$ from (\ref{eq def kappa}). We define $\fb=\Ad(\kappa)\fa$. Note that $\fb$ is a Cartan subalgebra of $\fg$ and that
$\fa_{\fh}\subseteq\fb\subseteq\fh$. Furthermore, $\fb$ is a Cartan subalgebra of $\fh$ as well.
We write $B$ for $\exp(\fb)$.

Recall the functionals $e_{k}$ from Section \ref{subsection Root systems in a}. For $1\leq k\leq n$ we define $f_{k}:\fb\to\R$ by $f_{k}=e_{k}\circ\Ad(\kappa)^{-1}$.
The root system of $\fb$ in $\fg$ then equals
$$
\hSigma:=\Sigma(\fg,\fb)
=\big\{f_{i}-f_{j}:1\leq i,j\leq n, i\neq j\big\}.
$$
The associated root spaces are given by $\fg_{f_{i}-f_{j}}=\R \big(\Ad(\kappa)E_{i,j}\big)$.

Note that
$$
\hSigma_{\fh}
:=\Sigma_{\fh}(\fb)
=\big\{f_{i}-f_{j}:1\leq i,j\leq n-1, i\neq j \big\}
$$
is both the root system of  $\fb$ in $\fh$ and  the set of $\fh$-roots in $\hSigma$. Let \begin{equation}
\label{e: pos system gS bar fh}
\hSigma_{\fh}^{+}
=\{f_{i}-f_{j}:1\leq i<j\leq n-1\big\}.
\end{equation}
Then $\hSigma_{\fh}^{+}$ is a positive system for $\hSigma_{\fh}$. Finally, we define
$$
\bar{\rho}_{\fh}
:=\frac{1}{2}\sum_{\alpha\in\hSigma_{\fh}^{+}}\alpha.
$$

%%% ----------------------------------------------------------------------
\subsection{Classification of parabolic subgroups}\label{subsection classification of parabolic subgroups for b}
%%% ----------------------------------------------------------------------

Recall that $\cQ$ denotes the set of minimal parabolic subgroups containing a maximal connected split abelian subgroup that is contained in $H$. Furthermore,
 $\cQ(\fb)$ denotes the subset of $\cQ$ consisting of minimal parabolic subgroups containing $B$. Given $Q \in \cP(\fb),$ we agree to use the abbreviation
$\hSigma(Q): = \Sigma(\fb;Q).$

\begin{Lemma}\label{lemma for SLn: P is H-conjugate to Q in cP(b) in standard position}
Let $P\in\cQ$. Then there exists a parabolic subgroup $Q\in\cQ(\fb)$
which is $H$-conjugate to $P$ and satisfies
\begin{equation}\label{eq Sigma_h subseteq Sigma(P) for b}
\hSigma_{\fh}^{+}\subseteq\hSigma(Q).
\end{equation}
Let $Q$ be any such parabolic subgroup. Then $\rho_{Q,\fh} = \bar \rho_\fh.$
Moreover, $P$ is $\fh$-compatible if and only if $Q$ is $\fh$-compatible (see Definition
\ref{d: fh compatible}).  
\end{Lemma}

\begin{proof}
Since $P \in \cQ,$ it follows from Corollaries \ref{Cor H-conjugacy of parabolic subgroups} and \ref{Cor Q in cP(A) is Nor_(K cap H)(fa)-conjugate to ps with Sigma_h^+ subseteq Sigma(P)}
that $P$ is $H$-conjugate to a minimal parabolic subgroup $Q \in\cQ(\fb)$
satisfying (\ref{eq Sigma_h subseteq Sigma(P) for b}). The latter condition
implies that $\rho_{Q,\fh} = \bar\rho_\fh.$ The final statement follows from
Remark \ref{r: h comp for H conjugation}.
\end{proof}

We will now classify the parabolic subgroups $Q\in\cQ(\fb)$ satisfying (\ref{eq Sigma_h subseteq Sigma(P) for b}).
The assignment $Q \mapsto \hSigma(Q)$ defines a bijection
from the set $\cP(\fb)$ onto the set of positive systems for $\hSigma.$ In turn, the latter set is in bijective correspondence with the permutation group $S_n.$
For a given $\tau \in S_n,$ the associated positive system is given by
$$
\hSigma(Q)
=\{f_{\tau^{-1}(i)}-f_{\tau^{-1}(j)}:1\leq i<j\leq n\}.
$$
Equivalently, a root $f_i - f_j \in \hSigma$ belongs to $\hSigma(Q)$
if and only if $\tau(i) < \tau(j).$ 
We infer that the parabolic subgroups $Q\in\cQ(\fb)$ satisfying (\ref{eq Sigma_h subseteq Sigma(P) for b})
correspond to the permutations $\tau\in S_{n}$ satisfying
\begin{equation}\label{eq standard tau ineq for b}
\tau(i)<\tau(j)\text{ for }1\leq i<j\leq n-1.
\end{equation}
Given such a permutation $\tau\in S_{n}$, there exists a unique $k$, with $1\leq k\leq n$, such that $\tau(n)=k$.
Conversely, for each integer $k$ with $1\leq k\leq n$, there exists a unique $\tau\in S_{n}$ satisfying (\ref{eq standard tau ineq for b}) and $\tau(n)=k$. From now on we write $Q_{k}$ for the corresponding minimal parabolic subgroup and $\hSigma^{+}_{k}$ for $\hSigma(Q_{k})$. Moreover, we write $N_{k}$ for $N_{Q_{k}}$.
For future reference we note that the positive system determined by $k$ is
given by the disjoint union
\begin{align}\label{eq Sigma_k for b}
\hSigma^{+}_{k}
=\hSigma_{\fh}^{+}
    \cup \big\{f_{i}-f_{n}:1\leq i\leq k-1\big\}
    \cup\big\{f_{n}-f_{i}:k\leq i\leq n-1\big\}.
\end{align}
We now provide $\fb^*$ with the inner product that turns
$\Ad(\kappa)^*: \fb^* \to \fa^*$ into an isometry;
see (\ref{e: inner product for roots on a}) for the description
of the inner product on $\fa^*$.

\begin{Lemma}\label{Lemma langle f_i-f_j,rho_h rangle}
Let $i$ and $j$ be integers such that $1\leq i<j\leq n$. Then
$$
\langle f_{i}-f_{j},\bar{\rho}_{\fh}\rangle=
\begin{cases}
j-i & \text{if }1\leq i<j\leq n-1,\\
\frac{n}{2}-i & \text{if }1\leq i\leq n-1\text{ and }j=n.
\end{cases}
$$
In particular $\langle\alpha,\bar{\rho}_{\fh}\rangle>0$ for every $\alpha\in\hSigma^{+}_{k}$ if and only if $\frac{n}{2}< k< \frac{n}{2}+1$, i.e., if and only if $n$ is odd and $k=\frac{n+1}{2}$.
\end{Lemma}

\begin{proof}
Using the definition of $\bar{\rho}_{\fh}$, we find
\begin{align*}
2\bar{\rho}_{\fh}
=\sum_{i=1}^{n-1}(n-2i)f_{i}.
\end{align*}
The first statement follows directly from this formula and the second follows from comparison with (\ref{eq Sigma_k for b}).
\end{proof}

Combining the previous lemmas, we now arrive at the following proposition.

\begin{Prop}\label{Prop Q H-comp for b iff H-conj to Q_k with k=(n+1)/2 }
Let $Q\in\cQ$. Then there exist a unique integer $k$, with $1\leq k\leq n$, such that $Q$ is $H$-conjugate to $Q_{k}$. Moreover, $Q$ is $\fh$-compatible if and only if $n$ is odd and $k=\frac{n+1}{2}$.
\end{Prop}

\begin{proof}
Only the uniqueness remains to be proved. Let $1\leq k,k'\leq n$ and assume that there exists an $h\in H$ such that $hQ_{k}h^{-1}=Q_{k'}$. Then $B$ and $hBh^{-1}$ are both $\sigma$-stable split components of $Q_{k'}$. From Theorem \ref{Thm P contains sigma-stable A, unique up to N cap H conjugation} it follows that there exists a unique $n\in N_{k'}\cap H$ such that $nhBh^{-1}n^{-1}=B$. Let $h'=nh$. Then $h'\in\Nor_{H}(\fb)$ and $h'Q_{k}h'^{-1}=Q_{k'}$.
Now $h'$ induces an element $w$ in the Weyl group of the root system
$\hSigma_{\fh}$.
Since $\hSigma_{\fh}^{+}$ is a positive system for $\hSigma_{\fh}$ and
$\hSigma_{\fh}^{+}$ is contained in both $\hSigma^{+}_{k}$ and $\hSigma^{+}_{k'}$, it follows that $w$ acts trivially on $\hSigma^{+}_{\fh}$ and hence that $h'$ acts trivially on $\fb$. We conclude that $hQ_{k}h^{-1}=Q_{k'}$ and therefore $k=k'$. This proves uniqueness.
\end{proof}

%%% ----------------------------------------------------------------------
\subsection{Decomposition and invariant measures}\label{subsection invariant measures for b}
%%% ----------------------------------------------------------------------
Let $k$ be an integer such that $1\leq k\leq n$.

\begin{Lemma}\label{Lemma n_k cap h for b}
Let $\fn_{k}$ be the Lie algebra of $N_{k}$. Then
\begin{align*}
\fn_{k}\cap\fh
    &=\bigoplus_{\alpha\in\hSigma_{\fh}^{+}}\fg_{\alpha}
    =\bigoplus_{1\leq i<j\leq n-1}\R \big(\Ad(\kappa)E_{i,j}\big),\\
\fn_{k}\cap\fq
    &=\bigoplus_{\alpha\in\hSigma\setminus\hSigma_{\fh}^{+}}\fg_{\alpha}
    =\bigoplus_{1\leq i\leq k-1}\R \big(\Ad(\kappa)E_{i,n}\big)
        \oplus\bigoplus_{k\leq i\leq n-1}\R \big(\Ad(\kappa)E_{n,i}\big).
\end{align*}
\end{Lemma}

\begin{proof}
Let $\alpha\in\hSigma$.
Since $\sigma\alpha=\alpha$ and $\fg_{\alpha}$ is $1$-dimensional, we have either $\fg_{\alpha}\subseteq \fh$ or $\fg_{\alpha}\subseteq \fq$. By definition, the first is the case for $\alpha\in\hSigma_{\fh}$ and the latter for $\alpha\in\hSigma\setminus\hSigma_{\fh}$. The lemma now follows from (\ref{eq Sigma_k for b}).
\end{proof}

We write $\Nq$ for the submanifold $\exp(\fn_{k}\cap \fq)$ of $N_{k}$.
For $x,y\in\R^{n-1}$ with $\langle x,y\rangle=0$ we define
\begin{equation}\label{eq def v_(x,y)}
v_{x,y}
=\kappa\exp\left(
   \begin{array}{c:c}
      0         &   x   \\
   \hdashline
      y^{t}     &   0   \\
   \end{array}
 \right)\kappa^{-1}
=\kappa\left(
   \begin{array}{c:c}
      I_{n-1}+\frac{1}{2}xy^{t} &   x   \\
   \hdashline
      y^{t}                     &   1   \\
   \end{array}
 \right)\kappa^{-1}.
\end{equation}
A straightforward computation shows that
$$
\Nq
=\big\{v_{x,y}
    :x\in\R^{k-1}\times\{0\}^{n-k},y\in\{0\}^{k-1}\times\R^{n-k}\big\}.
$$
We equip $\Nq$ with the push-forward along $\exp$ of the Lebesgue measure on $\fn_{k}\cap \fq$.
Then
the following lemma is  a direct consequence  of \cite[Prop.~1.1]{Benoist_AnalyseHarmoniqueSurLesEspacesSymetriquesNilpotents} and Lemma \ref{Lemma n_k cap h for b}.

\begin{Lemma}\label{Lemma decomposition of Haar measures for b}
The map
$$
\Nq\times (N_{k}\cap H)\to N_{k};\qquad(v,n)\mapsto vn$$
is a diffeomorphism. Moreover, there exists a 
normalization for the invariant measure $dx$ on $N_{k}/(N_{k}\cap H)$
such that for every $\phi\in C_{c}^{\infty}\big(N_{k}/(N_{k}\cap H)\big)$,
$$
\int_{N_{k}/(N_{k}\cap H)}\phi(x)\,dx
=\int_{\Nq}\phi\big(v\cdot(N_{k}\cap H)\big)\,dv.
$$
\end{Lemma}

%%% ----------------------------------------------------------------------
\subsection{Reduction to an equivalent theorem}\label{subsection equivalent theorem for b}
%%% ----------------------------------------------------------------------
Under the assumption that $P$ is of $\sigma$-parabolic rank $0$, it follows from Proposition \ref{Prop Q H-comp for b iff H-conj to Q_k with k=(n+1)/2 } and Lemma \ref{Lemma decomposition of Haar measures for b} that Theorem \ref{Theorem main theorem convergence} is equivalent to the following theorem.

\begin{Thm}\label{Thm alternative for main theorem for b}
Let $k$ be an integer with 
$1\leq k\leq n$.
Then the integral
$$
\int_{\Nq}\phi(v\cdot H)\,dv
$$
is absolutely convergent for every $\phi\in\cC(X_{n})$ if and only if
$k=\frac{n+1}{2}$. In particular, if $n$ is even, then for every $k$ there exists $\phi\in\cC(X_{n})$ such that the integral is divergent.
\end{Thm}
%%% ----------------------------------------------------------------------
\subsection{The integral}\label{subsection integral for b}
%%% ----------------------------------------------------------------------
Recall (\ref{eq def v_(x,y)}) for the definition of $v_{x,y}$.
\begin{Lemma}
\label{l: cosh 4 t is sum squares}
Let $x,y\in\R^{n-1}$ with $\langle x,y\rangle=0$. If $v_{x,y}\cdot H\in Ka_{t}\cdot H$, then
$$
2\cosh(4t)
=2+4\|x\|^{2}+4\|y\|^{2}+4\|x\|^{2}\|y\|^{2}.
$$
\end{Lemma}

\begin{proof}
By straightforward computations we see that
\begin{equation}\label{eq kappa v sigma(kappa v)^(-1)}
v_{x,y}\sigma(v_{x,y})^{-1}
=v_{x,y}^{2}
=v_{2x,2y}
=\kappa\left(
   \begin{array}{c:c}
      I_{n-1}+2xy^{t}           &   2x   \\
   \hdashline
      2y^{t}                    &   1   \\
   \end{array}
 \right)\kappa^{-1}
\end{equation}
and
$$
\|I_{n-1}+2yx^{t}\|_{HS}^{2}
=n-1+4\|x\|^{2}\|y\|^{2}.
$$
Therefore,
$$
\|v_{x,y}\sigma(v_{x,y})^{-1}\|_{HS}^{2}
=n+4\|x\|^{2}+4\|y\|^{2}+4\|x\|^{2}\|y\|^{2}.
$$
The lemma now follows from Lemma \ref{Lemma KAH decomposition}.
\end{proof}

\begin{Cor}\label{Cor integral equality for b}
Let $\phi\in C(X_{n})$ be $K$-invariant and non-negative. As in Corollary \ref{Cor integral equality},  let $\bar{\phi}: \R_{\geq2} \to \R_{\geq0}$ be defined by
$$
\bar{\phi}(2\cosh4t)=\phi(a_{t}\cdot H).
$$
Then
$$
\int_{\Nq}\phi(v\cdot H)\,dv =
\int_{\R^{k-1}}\int_{\R^{n-k}}
    \bar{\phi}\big(2+4\|x\|^{2}+4\|y\|^{2}+4\|x\|^{2}\|y\|^{2}\big)\,dy\,dx.
$$
\end{Cor}

\begin{proof}
Let $x \in \R^{k-1}\times \{0\}^{n-k},$ $y \in \{0\}^{k-1} \times \R^{n-k}$ and assume that
$v_{x,y} \in k' a_t H,$ for $k' \in K$ and $t \in \R.$ Then it follows that
$$
\phi(v_{x,y}\cdot H) = \bar \phi (2 \cosh 4t) = \bar\phi(2 + 4 \|x\|^2 + 4\|y\|^2 + 4 \|x\|^2 \|y\|^2),
$$
by Lemma \ref{l: cosh 4 t is sum squares}. \end{proof}

%%% ----------------------------------------------------------------------
\subsection{The case of convergence}\label{subsection convergence for b}
%%% ----------------------------------------------------------------------

In view of Remark \ref{r: Schwartz seminorms}, the following proposition implies the `if' 
part of Theorem
\ref{Thm alternative for main theorem for b}.

\begin{Prop}\label{Prop int phi convergent for b}
Assume $k=\frac{n+1}2$.
For $m\in\R$,  let $\phi_m: X_{n}\to\R$ be given by
$$
\phi_m(k'a_{t}\cdot H)
= (2 \cosh 4t)^{\frac{1-n}4}\big(1+\log(2\cosh 4t)\big)^{-m}
    \qquad (k'\in K, t\in\R).
$$
Then there exists $m\geq 0$ such that the integral
$\displaystyle\int_{\Nq}\phi_m(v\cdot H)\,dv$
is absolutely convergent.
\end{Prop}

\begin{proof}
Let $\bar \phi_m$ be defined in terms of $\bar \phi$ as in Corollary \ref{Cor integral equality for b}.
Then
$
\bar \phi_m(z) = z^{\frac{1-n}{4}} ( 1 + \log z)^{-m}
$
from which we see that $\bar \phi_m$ is a decrasing function
of $z \geq 1.$ Hence,
$$
\bar\phi_m (2 + 4\|x\|^2 + 4 \|y\|^2 + 4 \|x\|^2 \|y\|^2)) \leq \bar \phi_m(1 + \|x\|^2 + \|y\|^2 + \|x\|^2 \|y\|^2),
$$
for $x \in \R^{k-1} = \R^{\frac{n-1}2},$ $y\in \R^{n-k} = \R^{\frac{n-1}2}$
and by the mentioned corollary we see that
\begin{align*}
&\int_{\Nq[\frac{n+1}{2}]} \phi_{m} (v\cdot H)\,dv\\
&\quad \leq \int_{\R^{\frac{n-1}2}}\int_{\R^{\frac{n-1}2}} \;
     \bar \phi_m ( 1+\|x\|^{2}+\|y\|^{2}+\|x\|^{2}\|y\|^{2} )\, dy\, dx \\
&\quad = \int_{\R^{\frac{n-1}{2}} } \int_{\R^{\frac{n-1}{2}} }
    \frac{\big(1+\|x\|^{2}\big)^{\frac{1-n}4}\big(1+\|y\|^{2}\big)^{\frac{1-n}4}}
        {\big(1+\log(1+\|x\|^{2}+\|y\|^{2}+\|x\|^{2}\|y\|^{2})\big)^{m}}\;dy\,dx.
\end{align*}
Since 
$$
\big(1+\log(1+\|x\|^{2}+\|y\|^{2}+\|x\|^{2}\|y\|^{2})\big)^{2}
\geq \big(1+\log(1+\|x\|^{2})\big)\big(1+\log(1+\|y\|^{2})\big),
$$
the last double integral is at most
$$
\left(\int_{\R^{\frac{n-1}{2}}}
    \frac{\big(1+\|x\|^{2}\big)^{\frac{1-n}4}}
        {\big(1+\log(1+\|x\|^{2})\big)^{\frac{m}{2}}}\,dr\right)^{2}.
$$
By using polar coordinates, one readily verifies that the integral in this expression is absolutely convergent for $m>2$.
\end{proof}

%%% ----------------------------------------------------------------------
\subsection{The case of divergence}\label{subsection divergence for b}
%%% ----------------------------------------------------------------------

The following proposition and its corollary imply the `only if'  part
of Theorem
\ref{Thm alternative for main theorem for b}.

\begin{Prop}\label{Prop int phi divergent for b}
Let $\nu\in\R$ and let $\phi_{\nu}: X_{n}\to\R$ be given by
$$
\phi_{\nu}(k'a_{t}\cdot H)= ( 2\cosh 4t)^\nu
\qquad (k'\in K, t\in\R).
$$
Then the integral
$\displaystyle\int_{\Nq}\phi_{\nu}(v\cdot H)\,dv$
is divergent for $\nu\geq\min\big\{\frac{1-k}{2},\frac{k-n}{2}\big\}$.
\end{Prop}

\begin{proof}
The function $\bar \phi_\nu$ associated to $\phi_\nu$ as in Corollary \ref{Cor integral equality for b} is given by $z \mapsto z^\nu.$ By the mentioned corollary
we obtain
$$
\int_{\Nq}\phi_{\nu}(v\cdot H)\,dv
=\int_{\R^{k-1}}\int_{\R^{n-k}}
    \big(2+4\|x\|^{2}+4\|y\|^{2}+4\|x\|^{2}\|y\|^{2}\big)^{\nu}\,dy\,dx.
$$
Clearly the integral on the right-hand side is divergent if $\nu\geq0$. We assume that $\nu<0$. Then the integral on the right-hand side is larger than
\begin{align}
&4^{\nu}\int_{\R^{k-1}}\int_{\R^{n-k}}
    \big(1+\|x\|^{2}+\|y\|^{2}+\|x\|^{2}\|y\|^{2}\big)^{\nu}\,dy\,dx \nonumber \\
&\qquad=4^{\nu}    \int_{\R^{k-1}}     \big(1+\|x\|^{2}\big)^{\nu}\,dx
    \int_{\R^{n-k}}       \big(1+ \|y\|^{2}\big)^{\nu}\,dy.
    \label{e: product of two integrals new}
\end{align}
The condition on $\nu$ implies that $2 \nu\geq 1-k$ or $2\nu\geq k-n.$
By using polar coordinates, we see that one of the integrals in (\ref{e: product of two integrals new}) diverges. This completes the proof.
\end{proof}

\begin{Cor}\label{Cor int_(V) divergent}
Assume $k\neq\frac{n+1}{2}$. Then there exists a function $\phi\in\cC(X_{n})$ such that
$\displaystyle{\int_{\Nq}\phi(v\cdot H)\,dv}$ is divergent.
\end{Cor}

\begin{proof}
Assume $k\neq\frac{n+1}{2}.$  Then $\min\big\{\frac{1-k}{2},\frac{k-n}{2}\big\}<\frac{1-n}{4}\;$  and we may select $\nu$ such that
$$
\min\big\{\textstyle{\frac{1-k}{2}},\textstyle{\frac{k-n}{2}}\big\}
<\nu
<\frac{1-n}{4}.
$$
Then $\phi_{\nu}\in\cC(X_{n})$ by Lemma \ref{Lemma phi_nu is Schwartz}. The claim now follows from Proposition \ref{Prop int phi divergent for b}.
\end{proof}

%%% ----------------------------------------------------------------------
\section{Proof of Theorem \ref{Thm main theorem limit behavior}}\label{Section Limits}
%%% ----------------------------------------------------------------------
%%% ----------------------------------------------------------------------
\subsection{Reduction to an equivalent theorem}\label{Subsection limit theorem for Q_k}
%%% ----------------------------------------------------------------------
By Proposition \ref{Prop P H-conj to P_k,l; H-comp iff (n+1)/2 leq k,l leq (n+3)/2} and Proposition \ref{Prop Q H-comp for b iff H-conj to Q_k with k=(n+1)/2 } it suffices to prove the claims in Theorem \ref{Thm main theorem limit behavior}
for $P=P_{k,l}$ with $\frac{n+1}{2}\leq k=l\leq \frac{n+3}{2}$ and (for $n$ odd) $Q=Q_{\frac{n+1}{2}}$ only.

We recall the definition of $a_{t}$ from (\ref{eq def a_t}). An easy computation shows that
$$
\delta_{P_{k,k}}(a_{t})
= e^{t}. 
$$
It follows from Theorem \ref{Thm alternative for main theorem} that for all $\phi\in\cC(X_{n})$  the integral 
$$
\int_{U_{k,k}} \phi(a_{s}u\cdot H)\,du
$$
yields a well-defined function of $s\in\R$.
We are interested in the decay of this function, or more precisely,
of the modified function
$$
s\mapsto e^s\int_{U_{k,k}} \phi(a_{s}u\cdot H)\,du,
$$
and will prove the following result, which implies Theorem \ref{Thm main theorem limit behavior}.

\begin{Thm}\label{Thm limit theorem}
Assume $\frac{n+1}2\leq k\leq \frac{n+3}2$.
Let $\phi\in\cC(X_{n})$.
\begin{enumerate}[(i)]
\item If $n$ is even, then for every $N\in\N$ there exist  $c>0$ and $m\in\N$ such that for every $s\in\R,$
\begin{equation}\label{sup for n even}
\Big|e^s\int_{\Ns[k,k]} \phi(a_{s}u\cdot H)\,du\Big| \leq c(1+|s|)^{-N}\mu_{1,m}(\phi).
\end{equation}

\item If $n$ is odd, then for every $R\in\R$ and $N\in\N$ there exist a $c>0$ and $m\in\N$ such that for every $s\in\R$ with $s<R,$
    \begin{equation}\label{sup for n odd}
    \Big|e^s\int_{\Ns[k,k]} \phi(a_{s}u\cdot H)\, du\Big|\leq c(1+|s|)^{-N}\mu_{1,m}(\phi).
    \end{equation}
Furthermore, for $s$ moving in the other direction, there exists an element $\kappa_{0} \in K,$ independent of $\phi,$ such that 
    \begin{equation}\label{e: limit is integral}
    \lim_{s\to\infty} e^s\int_{\Ns[k,k]} \phi(a_{s}u\cdot H)\,du
    =
    \int_{\Nq[\frac{n+1}{2}]}\phi(\kappa_{0} v\cdot H)\,dv.
    \end{equation}
    In particular, the limit exists, and is non-zero as a function of
   $\phi.$
   \end{enumerate}
\end{Thm}

%%% ----------------------------------------------------------------------
\subsection{Proof of Theorem \ref{Thm limit theorem}}\label{Subsection Limits}
%%% ----------------------------------------------------------------------
We recall the definition of $u_{x,y,z}$ from (\ref{eq def u_(x,y,z)})
and start with a few lemmas.

\begin{Lemma}\label{Lemma 2cosh(4t)=(s,x,y,z)}
Let $s\in\R$, $x\in\R^{n-k}$, $y\in\R^{k-2}$ and $z\in\R$.
If $a_su_{x,y,z}\cdot H\in Ka_{t}\cdot H$, then
\begin{equation}\label{eq 2cosh(4t)=(s,x,y,z)}
2\cosh(4t)
=f_1+f_2 \|x\|^{2}+f_3 \|y\|^{2}+\|x\|^{2}\|y\|^{2},
\end{equation}
where
\begin{align}
&f_1=f_1(s,z)=e^{4s}(1-z)^{2}(1+z)^{2}+e^{-4s}+2z^2,\label{f1}\\
&f_2=f_2(s,z)=e^{2s}(1+z)^{2}+e^{-2s}\label{f2},\\
&f_3=f_3(s,z)=e^{2s}(1-z)^{2}+e^{-2s}.\label{f3}
\end{align}
\end{Lemma}

\begin{proof}  From (\ref{eq def u_(x,y,z)}) we find
$$
a_su_{x,y,z}=
\left(
   \begin{array}{c:c:c}
    e^s   &   e^s x^{t}   &   e^s z   \\
   \hdashline
        &   I_{n-2} &   y   \\
   \hdashline
        &           &   e^{-s}   \\
   \end{array}
 \right),
$$
hence
$$
a_{s}u_{x,y,z}\sigma(a_{s}u_{x,y,z})^{-1}
=
\left(
   \begin{array}{c:c:c}
   e^{2s}(1-z^{2}) &   e^{s}(1+z)x^{t}  &   z           \\
   \hdashline
   e^{s}(1-z)y     &   I_{n-2}+yx^{t}   &   e^{-s}y     \\
   \hdashline
   -z              &   e^{-s}x^{t}      &   e^{-2s}     \\
   \end{array}
\right).
$$
The proof is completed by combining Lemma \ref{Lemma KAH decomposition}
with
 a straightforward computation of the squared Hilbert-Schmid norm
of the last matrix,
analogous to the computation in the proof of Lemma \ref{Lemma 2cosh(4t)=P(x,y,z)}
(note that now $\langle x,y\rangle=0$ since $k=l$).
\end{proof}

We shall need some estimates for
$f_1$, $f_2$ and $f_3$.

\begin{Lemma}\label{Lemma f1f2f3}
Let $s, z\in\R$ and let $f_i=f_i(s,z)$
be as above for $i=1,2,3$. Then
\begin{equation}\label{1st bound}
f_2f_3=f_1+2,\quad 2\leq f_1, \quad
f_1\leq f_2f_3\le 2f_1
\end{equation}
and
\begin{equation}\label{2nd bound}
1+z^2\le f_1.
\end{equation}
\end{Lemma}

\begin{proof} The equality in (\ref{1st bound})
is easily verified, and the lower bound for $f_1$
follows from (\ref{eq 2cosh(4t)=(s,x,y,z)}) with
$x=y=0$. Then,
$f_1\leq f_1+2\leq 2f_1$
implies the final statement in (\ref{1st bound}).
Finally we observe that in addition to $2\le f_1$
we also have $2z^2\le f_1$, whence (\ref{2nd bound}).
\end{proof}

\begin{Lemma} Let $R\in\R$.
Then there exists a constant $A>0$
such that
\begin{equation}\label{fi ineq}
f_i(s,z)\ge A(e^{2s}z^2+e^{-2s}),\quad i=2,3,
\end{equation}
for all $z\in\R$ and all $s\le R$.
\end{Lemma}

\begin{proof}
We may assume $i=2$ since $f_3(s,z)=f_2(s,-z)$.
Let $b\in\R$ be the solution to $b^2-b=e^{-4R}$ that is larger than $1$.
We shall establish (\ref{fi ineq}) for all $s\le R$
with
$$A=\frac{b-1}{b}.$$
Inserting the definition of
$f_2$ we see that with this value of $A$,
(\ref{fi ineq}) is equivalent to
$$e^{2s}(z+1)^2+e^{-2s}\ge \frac{b-1}b (e^{2s}z^2+e^{-2s})$$
and hence also to
\begin{equation*}
e^{2s}(z^2+2b z+b)+e^{-2s}\ge 0.
\end{equation*}
This last inequality is valid for all $z\in\R$ and $s\le R$
since the minimum of $z^2+2b z+b$ as a function of $z$ is
$-b^2+b=-e^{-4R}$.
\end{proof}

\begin{proof}[Proof of Theorem \ref{Thm limit theorem}]
Let $\phi\in\cC(X_{n}).$ Throughout the proof we will use the notation
$$
\Iphi(s): = \int_{U_{k,k}} \phi(a_s u) \; du.
$$
The proof consists of three parts. In part (a) we will
address the rapid decay of $e^s\Iphi(s)$ for $s\to-\infty$ both for $n$ odd and even. In part (b) we will
address the similar decay for $s\to \infty$ in case  $n$ is even. Finally, in part (c) we will
address the limit behavior for $s \to \infty$ in case $n$ is odd.

Fix $N \in \N$. Then by (\ref{e: seminorm mu u m}),
the function
$\phi$ satisfies the estimate
\begin{equation}
\label{e: phi estimated on seminorm times f}
|\phi(x)|
\leq\mu_{1,N}(\phi)f(x)
\qquad(x\in X_{n}),
\end{equation}
where $f:X_{n}\to\R_{>0}$ is given by 
\begin{equation}\label{eq decay of f}
f(k'a_{t}\cdot H)
= \big(2\cosh(4t)\big)^{-\frac{n-1}{4}}\big(1+\log\big(2\cosh(4t)\big)\big)^{-N}
\qquad(k'\in K, t\in\R).
\end{equation}
Let
$$
\If(s):=\int_{\Ns[k,k]} f(a_{s}u\cdot H)\,du,
$$
then we have the estimate
$$
|\Iphi(s)| \leq \mu_{1,N}(\phi) I_f(s),
$$
so that for parts (a) and (b) it suffices to show
that (\ref{sup for n even}) and (\ref{sup for n odd}) are satisfied with $\phi$ replaced by $f$.
Define $\Psi: \R_{\geq 1} \to \R$ by
\begin{equation}\label{est Psi new}
\Psi(r)  =   r^{-\frac{n-1}4}(1+\log(r))^{-N},
\quad(r\ge 1, N\ge 0).
\end{equation}
Then it follows that 
$$
\Psi(2\cosh(4t))= f(k' a_{t}\cdot H)
$$
for $k' \in K$ and $t\in\R$.
Moreover, using (\ref{eq 2cosh(4t)=(s,x,y,z)}) we see that
\begin{equation}
\label{eq: estimate phi a u in terms of Psi}
f(a_s u_{x,y,z}) =  \Psi(f_1+f_2\|x\|^{2}+f_3\|y\|^{2}+\|x\|^{2}\|y\|^{2}).
\end{equation}
Hence,
$$
\If(s)=\int_{z\in\R}\int_{y\in\R^{k-2}}\int_{x\in\R^{n-k}}
\Psi(f_1+f_2\|x\|^{2}+f_3\|y\|^{2}+\|x\|^{2}\|y\|^{2})\,dx\,dy\,dz.
$$

{\em Part (a).\ }
Performing the following substitutions on the inner integrals,
$$x=(f_1/f_2)^{1/2}\xi,\quad y=(f_1/f_3)^{1/2}\eta,$$
we obtain
\begin{equation}\label{eq I(s)}
\If(s)=\int_{z\in\R}\int_{\eta\in\R^{k-2}}\int_{\xi\in\R^{n-k}}
\Psi(F(\xi,\eta)f_1)
\big(\frac{f_1}{f_2}\big)^{\frac{n-k}2}\big(\frac{f_1}{f_3}\big)^{\frac{k-2}2}
\,d\xi\,d\eta\,dz,
\end{equation}
where
$$
F(\xi,\eta):=1+\|\xi\|^{2}+\|\eta\|^{2}+\frac{f_1}{f_2f_3}\|\xi\|^{2}\|\eta\|^{2}.
$$
In the following we assume that $N\ge 6$. Using (\ref{est Psi new}) 
we now see that the integrand in (\ref{eq I(s)}) is absolutely bounded by
\begin{equation}\label{an upper bound}
F(\xi,\eta)^{-\frac{n-1}4}
f_1^{\frac{n-3}4}
f_2^{-\frac{n-k}2}
f_3^{-\frac{k-2}2}
\big(1+\log F(\xi,\eta)+\log(f_1)\big)^{-N}.
\end{equation}
Let
$$
\epsilon=\textstyle{\frac{n+3}2}-k;
$$
then $0\leq \epsilon\leq 1$ by our assumption on $k$.
Moreover,
\begin{equation}\label{eps eqs}
n-k=\textstyle{\frac{n-3}2} +\epsilon,\qquad k-2=\textstyle{\frac{n-3}2}+1-\epsilon.
\end{equation}
Since $f_1\leq f_2f_3$, it follows that the expression
(\ref{an upper bound}) is bounded from above by
$$
F(\xi,\eta)^{-\frac{n-1}4}
f_2^{-\frac\epsilon2}f_3^{-\frac{1-\epsilon}2}
\big(1+\log F(\xi,\eta)+\log(f_1)\big)^{-N}
$$
for all $s$.

From (\ref{1st bound}) we infer that
$$
F(\xi,\eta)
\geq(1+\tfrac12\|\xi\|^{2})(1+\tfrac12\|\eta\|^{2}).
$$
It thus finally follows that
\begin{equation}\label{triple product}
\If(s) \leq  I_1 I_2 I_3(s),
\end{equation}
where
\begin{equation}\label{I1I2I3}
I_1=\int_{\R^{n-k}}
g_1(\xi)
\,d\xi,\quad
I_2=\int_{\R^{k-2}}
g_2(\eta)
\,d\xi,\quad
I_3(s)=\int_\R
g_3(s,z)\,dz,
\end{equation}
with
\begin{align}
g_1(\xi)=&(1+\tfrac12\|\xi\|^{2})^{-\frac{n-1}4}
\big(1+\log(1+\tfrac12\|\xi\|^{2})\big)^{-2},
\nonumber\\
g_2(\eta)=
&(1+\tfrac12\|\eta\|^{2})^{-\frac{n-1}4}
\big(1+\log(1+\tfrac12\|\eta\|^{2})\big)^{-2},
\label{I integrands}\\
g_3(s,z)=
& 
f_2(s,z)^{-\frac\epsilon2}
f_3(s,z)^{-\frac{1-\epsilon}2}
\big(1+\log(f_1(s,z))\big)^{-(N-4)}.
\nonumber
\end{align}
It follows from (\ref{eps eqs}) that the dimensions
$n-k$ and $k-2$ are at most $\frac{n-1}2$ so that
$I_1$ and $I_2$ are finite, thanks to the logarithmic terms
(which in fact are needed in at most one of the integrals).
Thus it only remains to estimate the third integral in (\ref{I1I2I3}).

We first assume $s\le R$ for some given $R\in\R$.
Using (\ref{fi ineq}) for $f_2$ and $f_3$, and estimating
two of the logarithmic factors in $I_3(s)$
by (\ref{2nd bound}) and the remaining ones by
$f_1\ge\max\{2,e^{-4s}\}$,
we find
$$I_3(s)\leq
C\int_\R
(e^{2s}z^2+e^{-2s})^{-1/2}
\big(1+\log(1+z^2)\big)^{-2}\,dz\,
\big(1+\max\{\log2,-4s\}\big)^{-(N-6)}
$$
for all $s\leq R$, with $C>0$ a constant
depending on $N$.
By substitution of $z=e^{-2s}\zeta$,
and using that $s\le R$,
we find
$$
I_3(s)\le
Ce^{-s}\int_\R
(\zeta^2+1)^{-1/2}
\big(1+\log(1+e^{-4R}\zeta^2)\big)^{-2}\,d\zeta\,
\big(1+|s|\big)^{-(N-6)},
$$
with a new constant $C>0$.
The integral converges, and since $N$ was arbitrary
we conclude from (\ref{triple product})
that (\ref{sup for n odd})
holds, regardless
of the parity of $n$. This completes part (a) of the proof.

{\em Part (b).\ } We assume that $s\ge 0$ and that $n$ is even.
Then $k=\frac{n+2}2$ and
$\epsilon=\frac12$,
hence
$$
I_3(s)=\int_\R
f_2(s,z)^{-\frac14}
f_3(s,z)^{-\frac14}
\big(1+\log(f_1(s, z))\big)^{-(N-4)}\,dz.
$$
The integral over $\R$ can be replaced by an integral
over $\R_{>0}$, because $f_2(s, -z)=f_3(s, z)$ and 
$f_1(s, -z ) = f_1(s,z ).$ We split
the integration into two parts, and integrate
separately over the interval $[1-\delta,1+\delta]$
and its complement in $\R_{>0}$,
with $\delta\in (0,1)$ to be fixed later
(it will depend on $s$). Let us write
$\firstJ_{\delta}(s)$ for the integral over $[1-\delta,1+\delta]$
and $\secJ_{\delta}(s)$ for the integral over the complement
of this set in  $\R_{>0}$.

For $\firstJ_\delta(s)$  we use the estimates
$$
f_1\ge 2,
\quad f_2(s,z)\ge e^{2s},
\quad f_3(s,z)\ge e^{2s}(z-1)^2,
$$
and for $z\ge 0$ and obtain
\begin{equation}\label{Idelta integral}
\firstJ_{\delta}(s)
\leq e^{-s}\int_{1-\delta}^{1+\delta}|z-1|^{-1/2}\,dz
=4e^{-s}\delta^{1/2}.
\end{equation}

For $\secJ_\delta(s)$ we estimate $f_1$ by (\ref{2nd bound})
in two of the logarithmic factors and by $f_1\geq e^{4s}\delta^2$
in the remaining factors. Furthermore,
we estimate
$$
f_2(s,z)\geq e^{2s}(1+z)^2,\quad
f_3(s,z)\geq e^{2s}(z-1)^2,$$
and obtain
$$
\secJ_\delta(s)
\le Ce^{-s}
\int_0^\infty
(1+z)^{-1/2}|z-1|^{-1/2}
\big(1+\log(1+z^2)\big)^{-2}\,dz
\big(1+\log(e^{4s}\delta^{2})\big)^{-(N-6)},
$$
with a constant $C>0$ depending on $N$ but
independent of $s$ and $\delta$.
The integral over $z$ converges and we obtain (with
a new constant $C>0$ of the same (in)dependency),
\begin{equation}\label{Jdelta integral}
\secJ_\delta(s)
\le Ce^{-s}
\big(1+\log(e^{4s}\delta^{2})\big)^{-(N-6)}.
\end{equation}
By adding
(\ref{Idelta integral}) and
(\ref{Jdelta integral}), we
see that by choosing $\delta=e^{-\frac{3}{2}s}$
we can ensure that
$$
I_3(s)\le Ce^{-s} (1+s)^{-(N-6)},
$$
with yet another constant $C>0$.
This implies (\ref{sup for n even}) for the remaining case $s\ge0$.

{\em Part (c).\ } We now turn to the statements about the limit in Theorem \ref{Thm limit theorem}. Assume that $n$ is odd. We shall first deal with the case $k=l=\frac{n+1}{2}$ and consider the integral
\begin{align}
\nonumber
\Iphi(s)
& = \int_{\Ns[\frac{n+1}{2},\frac{n+1}{2}]}\phi(a_{s}u\cdot H)\,du\\
\label{eq integral of phi(a_s n) new}
&= \int_{z\in\R}\int_{y\in\R^{\frac{n-3}{2}}\times\{0\}^{\frac{n-1}{2}}}
\int_{x\in\{0\}^{\frac{n-3}{2}}\times\R^{\frac{n-1}{2}}}\phi (a_{s}u_{x,y,z})\,dx\,dy\,dz.
\end{align}
This time, we perform the substitution of variables
\begin{equation}
\label{e: substitution x y z}
x = e^{s} \xi, \quad y = e^{-s} \eta, \quad { \rm and}
\quad  z = e^{-2s} \omega - 1,
\end{equation}
and obtain from (\ref{eq integral of phi(a_s n) new}) that
\begin{equation}\label{e: transformed integral I(s)}
 e^s \Iphi(s)
=\int_{\omega\in\R}\int_{\eta\in\R^{\frac{n-3}{2}}\times\{0\}^{\frac{n-1}{2}}}
    \int_{\xi\in\{0\}^{\frac{n-3}{2}}\times\R^{\frac{n-1}{2}}}
    \Phi_{s}(\xi,\eta,\omega)\,d\xi\,d\eta\,d\omega,
\end{equation}
where
$$
\Phi_{s}(\xi,\eta,\omega) = \phi(a_s u_{e^s\xi , e^{-s}\eta, e^{-2s} \omega - 1}).
$$
Recall the definition of $\kappa$ from (\ref{eq def kappa}).
From Lemma \ref{Lemma Limit of N orbits} below we see that
$$
\lim_{s \to \infty} \Phi_s(\xi, \eta, \omega) = \phi(\kappa^{-1} v_{(\omega, \eta), \frac23\xi }),
$$
for all 
$(\omega, \eta , \xi) \in \R^{n-1}.$ Assuming that we may interchange the limit
for $s \to \infty$ with the integral on the right-hand side of (\ref{e: transformed integral I(s)}) we obtain

\begin{eqnarray*}
\lim_{s \to \infty} e^s \Iphi(s) & = &
\int_\R
\int_{\R^{\frac{n-3}{2}}}
\int_{\R^{\frac{n-1}{2}}} \phi(\kappa^{-1} v_{(\omega,\eta), \frac23 \xi} )
\; d \xi \,d\eta\, d\omega \\
& = & \int_{V_{\frac{n+1}{2}}} \;  \phi(\kappa^{-1} v) \; dv,
\end{eqnarray*}
for the choice of Lebesgue measure $dv$ corresponding to $(2/3)^{\frac{n-1}{2}} d \xi \,d\eta\, d\omega.$

Thus, for the proof of (\ref{e: limit is integral}) it remains to be shown
that we may interchange limit and integral in (\ref{e: transformed integral I(s)}).  To prove this,
we adopt the following strategy.

For $0 < \gd < 1$  and $s > 0$ we define the set
$$
A_{\delta,s}:=\R^{\frac{n-1}{2}}\times \R^{\frac{n-3}{2}} \times [(2-\delta)e^{2s},(2+\delta)e^{2s}]
$$
and denote by $B_{\delta,s}$ its complement in $\R^{n-1} \simeq \R^{\frac{n-1}{2}}\times \R^{\frac{n-3}{2}}  \times \R$. We observe that for every $v \in \R^{n-1}$ there exists 
$s_ 0 \in \R$ such that $v \in B_{\delta,s}$ for all $0 < \gd < 1$ and $s \geq s_0.$
Accordingly, the characteristic function $1_{B_{\gd,s}}$ converges
to the constant function $1,$ pointwise on $\R^{n-1},$ for $s \to \infty.$

In the text below, we will show that
\begin{equation}
\label{e: estimate int over A gd s}
\int_{A_{\gd, s}}|\Phi_s(\xi, \eta, \omega)|\; d\xi \, d\eta \, d\omega \leq C \gd
\end{equation}
for a suitable constant $C > 0,$ independent of $s$ and $\gd.$
On the other hand, we will show that for each $0 < \gd < 1$ there exists
an integrable function $F_\gd:  \R^{n-1} \to \R_{\geq 0},$
such that
\begin{equation}
\label{e: estimate Phi s by Phi gd}
1_{B(\gd,s)} | \Phi_s | \leq F_\gd \quad{\rm on}\;\;\R^{n-1}
\end{equation}
for all $s > 0.$
By application of Lebesgue's convergence theorem
it then follows that
$$
\lim_{s \to \infty} \int_{B_{\gd, s}} \Phi_s \; d\xi \, d\eta \, d \omega =
\int_{\R^n} \lim_{s\to \infty} \Phi_s \; d\xi\,  d\eta \, d \omega.
$$
Combining this with (\ref{e: estimate int over A gd s}) we readily see that the interchange of
limit and integral is allowed.

To achieve the goals mentioned above,  we  recall the definition of
$f_1, f_2$ and $f_3$ from (\ref{f1}) - (\ref{f3}), but now
considered as functions of $(s, \omega).$
Then by virtue of the substitution (\ref{e: substitution x y z}),
if follows from
(\ref{f2}) that 
\begin{equation}
\label{e: f in terms of omega}
f_2 = e^{-2s}(\omega^2+1).
\end{equation}

Furthermore, we define $f$ and $\Psi$ as in
(\ref{eq decay of f}) and (\ref{est Psi new}). 
Then from (\ref{e: phi estimated on seminorm times f}) and (\ref{eq: estimate phi a u in terms of Psi}) we infer that
$$
|\Phi_s(\xi, \eta, \omega)| \leq \mu_{1, N}(\phi)\,   \Psi(f_1 + e^{2s} f_2 \|\xi\|^2 + e^{-2s} f_3 \|\eta\|^2  + \|\xi\|^2 \|\eta\|^2).
$$
Since $\Psi$ is decreasing on $\R_{\geq 1},$ whereas $f_3 \geq f_1/f_2$ by
(\ref{1st bound}), it follows that
\begin{equation}
\label{e: estimate Phi s by Psi s}
|\Phi_s(\xi, \eta, \omega) |\leq \mu_{1,N}(\phi)\, \Psi_s(\xi, \eta, \omega),
\end{equation}
where
$$
\Psi_s(\xi, \eta, \omega) := \Psi(f_1 + e^{2s}f_2\|\xi\|^2 + \frac{f_1}{e^{2s} f_2}\|\eta\|^2 + \|\xi\|^2\|\eta\|^2).
$$
This estimate, combined with (\ref{e: f in terms of omega}), motivates the use of a final substitution
$$
\xi=(1 + \omega^{2})^{-\frac{1}{2}}\chi,
\quad
\eta=(1 + \omega^{2})^{\frac{1}{2}}\psi.
$$
Note that this substitution does not effect the subsets $A(\gd, s)$ and $B(\gd, s)$ of $\R^{n-1},$ defined above. For any measurable subset $S \subseteq \R,$
 the function $\Psi_s$  is
integrable over $\R^{n-2} \times S$ if and only if the function
$$
\widetilde \Psi_s(\chi, \psi, \omega):=  (1 + \omega^2)^{-\frac12}\;\Psi(f_1 + \|\chi\|^2 + f_1 \|\psi\|^2 + \|\chi\|^2 \|\psi\|^2)
$$
is integrable over this set, and accordingly,
$$
\int_{\R^{n-2} \times S} \Psi_s(\xi,\eta,\omega)\; d\xi\, d\eta\,d\omega =
\int_{\R^{n-2} \times S} \widetilde\Psi_s(\chi,\psi,\omega)\; d\chi \, d\psi \,d\omega.
$$
We observe that
$$
1 + \|\chi\|^2 + \|\psi\|^2 + \|\chi\|\|\psi\| \geq (1+ \frac12 \|\chi\|)( 1 + \frac12 \|\psi\|).
$$
Since $\Psi: r \mapsto r^{-\frac{n-1}4} ( 1 + \log r)^{-N}$ is decreasing and $f_1 \geq 2,$ we now obtain the estimate
\begin{eqnarray*}
\widetilde \Psi(\chi, \psi, \omega) &  \leq  & (1 + \omega^2 )^{-1/2}
\Psi(f_1 + \|\chi\|^2 +  \|\psi\|^2 + \|\chi\|^2 \|\psi\|^2)\\
&\leq &
g_1(\chi)\,g_2(\psi)\,h_s(\omega),
\end{eqnarray*}
where $g_1,g_2$ are defined as in (\ref{I integrands}) and where
$$
h_s(\omega) := (1 + \omega)^{-\frac12} ( 1+ \log f_1)^{-(N-4)}.
$$
We observed already that the functions $g_1$ and $g_2$ are integrable over $\R^{\frac{n-1}2}$ and $\R^{\frac{n-3}2},$ with integrals $I_1$ and $I_2,$ respectively. Let us therefore
 focus on the function $h_s.$ On the interval $[(2 -\gd)e^{2s}, (2+\gd)e^{2s}]$ we have  the estimates $1 + \omega^2 \geq e^{4s}$
 and $f_1 \geq 1.$ Hence,
 $$
 \int_{I_{\gd , s}} h_s(\omega) \; d\omega \leq
 \int_{(2-\delta)e^{2s}}^{(2+\delta)e^{2s}}e^{-2s}\,d\omega= 2 \delta.
$$
It follows that
$$
\int_{A_{\gd, s}} \Psi_s(\xi, \eta, \omega) \; d\xi\,d\eta\, d\omega \leq 2 I_1 I_2 \gd
$$
and by applying (\ref{e: estimate Phi s by Psi s})  we obtain the estimate (\ref{e: estimate int over A gd s}).

It remains to prove the claimed majorization of $\Phi_s$ on $B_{\delta,s}$.
Here we shall use the following lower bound on $f_1,$
\begin{equation}\label{e: lbd}
f_1\ge 1+\frac12\delta^2\omega^2 \qquad (  \omega\notin [(2-\delta)e^{2s},(2+\delta)e^{2s}]) .
\end{equation}
To see this, note that the condition on $\omega$ is equivalent to $|2-e^{-2s}\omega|\ge \delta$, and hence (\ref{e: lbd}) follows immediately from the estimates $f_{1}\geq\frac{1}{2}f_{2}f_{3}$ and $f_{1}\geq 2$ (see (\ref{1st bound})).

From (\ref{e: lbd}) we obtain the estimate
$$
h_s(\omega) \leq (1 + \omega^2)^{-\frac12} (  1 + \frac12 \gd^2 \omega^2)^{-(N-4)}\qquad
(\omega \in \R \setminus [(2-\delta)e^{2s},(2+\delta)e^{2s}])
$$
for every $s >0.$ We now make the additional assumption that $N \geq 5$  to ensure
that the function on the right-hand side is integrable over $\R.$

Define the function $ \widetilde G_\gd : \R^{n-1} \to \R_{\geq 0}$ 
by
$$
\widetilde G_\gd (\chi, \psi,\omega) = g_1(\chi)\,g_2(\psi)\,(\omega^2+1)^{-\frac12}\big(1+\log(1+\frac{1}{2}\delta^2\omega^2)\big)^{-(N-4)}.
$$
Then $\widetilde G_\gd$  is integrable on $\R^{n-1}$ and for every $s > 0$ we have the estimate
$$
\widetilde \Psi_s \leq \widetilde G_\gd \qquad {\rm on} \;\; B_{\gd, s}.
$$
Define $G_\gd : \R^{n-1} \to \R_{\geq 0}$  by
$$
G_\gd (\xi, \eta, \omega) = (1 + \omega^2)^{\frac12} \widetilde G_\gd ((1 + \omega^2)^{\frac12} \xi,
(1 + \omega^2)^{- \frac12} \eta, \omega).
$$
Then $G_\gd $ is integrable on $\R^{n-1}$, and it follows that
$$
\Psi_s \leq G_\gd  \qquad {\rm on} \;\; B_{\gd, s},
$$
for every $s > 0.$ In view of (\ref{e: estimate Phi s by Psi s}) this proves
(\ref{e: estimate Phi s by Phi gd}) with $F_\gd  := \mu_{1,N}(\phi) G_\gd.$
We have thus established the limit formula (\ref{e: limit is integral}) for the case $k=l=\frac{n+1}{2}$.

We will complete the proof of Theorem \ref{Thm limit theorem} by proving (\ref{e: limit is integral}) for the remaining case $k=l=\frac{n+3}{2}$.
Let $w$ be the longest Weyl group element for the
root system $\Sigma_{\fh}=\Sigma\big(\Cen_{\fh}(\fa_{\fq});\fa_{\fh}\big),$ relative
to the positive system (\ref{e: positive system gS fh}), and let $w_{0} \in N_{K\cap H}(\fa) \cap \Cen_{K\cap H}(\fa_{\fq})$
be a representative for $w$. By Lemma \ref{Lemma equivalences between parabolics},
\begin{align}
\lim_{s\to\infty}e^{s}\int_{\Ns[\frac{n+3}{2},\frac{n+3}{2}]}\phi(a_{s}u)\,du
&=\lim_{s\to\infty}e^{s}\int_{\Ns[\frac{n+1}{2},\frac{n+1}{2}]}\phi\big(w_{0}\,\sigma\theta(a_{s}u)\big)\,du\nonumber \\
&=\int_{\Nq[\frac{n+1}{2}]}\phi\big(w_{0}\,\sigma\theta(\kappa^{-1} v)\cdot H\big)\,dv\nonumber\\
&=\int_{\theta \Nq[\frac{n+1}{2}]}\phi\big(w_{0}\,\kappa \bar{v}\cdot H\big)\,d\bar{v}.
\label{e: rewritten limit of int}
\end{align}
Let $w_{1}$ be a representative in $N_{K\cap H}(\fb)$
 for the longest Weyl group element of $\Sigma_{\fh}(\fb),$ relative to the positive system
(\ref{e: pos system gS bar fh}).
 Then,   in view of Proposition
 \ref{Prop Q H-comp for b iff H-conj to Q_k with k=(n+1)/2 },
$$
w_{1}\theta \Nq[\frac{n+1}{2}]w_{1}^{-1}
=\Nq[\frac{n+1}{2}].
$$
Hence, the integral in  (\ref{e: rewritten limit of int})  is equal to
$$
\int_{\Nq[\frac{n+1}{2}]}\phi\big(w_{0}\kappa w_{1}^{-1} v\cdot H\big)\,dv.
$$
This completes the proof of Theorem \ref{Thm limit theorem}.
%\hbox{\vspace{-30pt}}{\ }
\end{proof}
\medbreak
\begin{Lemma}\label{Lemma Limit of N orbits}
Let $k = l = \frac{n+1}{2}.$ Then with notation as in
(\ref{eq def u_(x,y,z)}) and (\ref{eq def v_(x,y)})
and writing $u(x,y,z) = u_{x,y,z},$
we have the following limit in $X_n,$
$$
\lim_{s \to \infty} a_s \, u(e^{s} \xi, e^{-s} \eta, e^{-2s}\omega - 1)\cdot H = \kappa^{-1} v_{(\omega, \eta), \frac23 \xi}  \cdot H.
$$
\end{Lemma}

\begin{proof}
We start with a computation in $\SL(2,\R),$ whose Lie algebra has the standard
Iwasawa decomposition
$
\fsl(2,\R) =  \R U \oplus \R Y \oplus \R V,$
with
$$
U: = \left(\begin{array}{rr} 0 & -1 \\ 1 & 0\end{array}\right), \quad
Y: = \left(\begin{array}{rr} 1 & 0  \\ 0 & -1\end{array}\right), \quad
V: = \left(\begin{array}{rr} 0 & 1 \\  0 & 0\end{array}\right).
$$
We agree to write
$
k_\gf := \exp \gf U,$ $ a_s:= \exp s Y $ and $n_z := \exp z V.$ In particular,
$$
\kappa := k_{\frac\pi4} = \frac1{\sqrt 2} \left(\begin{array}{rr} 1 & - 1 \\ 1 &  1\end{array}\right)
\quad{\rm and} \quad
n_{-1} = \left(\begin{array}{rr} 1 & - 1 \\ 0  & 1 \end{array}\right).
$$
Let $\gs$ be the involution of $\SL(2,\R)$ given by switching the diagonal entries
as well as the off diagonal entries, and let $H_0$ denote the
associated group of fixed points.

We will first show that in the quotient $\SL(2,\R)/H_0$ we have
\begin{equation}
\label{e: limit a s n -1}
\lim_{s \to \infty}  a_s n_{-1} \cdot H_{0} = \kappa^{-1} \cdot H_{0}.
\end{equation}
According to \cite[Thm.~1.3]{Loos_SymmetricSpacesI} the map
$G \to G,$ $g \mapsto g \gs(g)^{-1}$ induces a diffeomorphism from
$G/G^\gs$ onto a submanifold of $G.$  Applying this general fact
to the situation at hand, we see that
for (\ref{e: limit a s n -1}) to be valid, it
suffices to show that
$$
a_s n_{-1} \gs(a_s n_{-1})^{-1} \to \kappa \gs(\kappa)^{-1} = k_{\frac\pi 2}\qquad (s \to \infty).
$$
Now this follows by a straightforward calculation.

In turn, it follows from (\ref{e: limit a s n -1}) that there exists
a function $q: \R \to \SL(2,\R)$ with $\lim_{s \to \infty} q(s) = e$ and
$$
n_1 a_{-s} \kappa^{-1} \, q(s)  \in H_0
\qquad (s \in \R).
$$
We agree to identify $\SL(2,\R)$ with a closed subgroup of $\SL(n,\R)$ via the embedding
$$
\left(\begin{array}{cc}
a & b \\
c & d
\end{array}
\right)\;\;
\mapsto \;\;
\left(
   \begin{array}{c:c:c}
    a   &      &   b   \\
   \hdashline
        &   I_{n-2} &      \\
   \hdashline
    c    &           &   d   \\
   \end{array}
 \right).
$$
Then $\kappa$ and $a_s$ in $\SL(2,\R)$ correspond with the similarly denoted
elements in $\SL(n, \R).$ Furthermore, $H_0$ equals the intersection of $\SL(2,\R)$
with the subgroup $H$ of $\SL(n,\R).$

For $x \in \{0\}^{\frac{n-3}2} \times \R^{\frac{n-1}2},$
$y \in \R^{\frac{n-3}2} \times \{0\}^{\frac{n-1}2} $ and $z \in \R$ we define
\begin{equation}
\label{e: defi w}
w(x,y,z) := \exp
\left(
   \begin{array}{c:c:c}
       &      &   z   \\
   \hdashline
        &    &   y   \\
   \hdashline
        &     x^t      &     \\
   \end{array}
 \right) =
 \left(
   \begin{array}{c:c:c}
    1   &  \frac12 z x^t     &   z   \\
   \hdashline
        &   I_{n-2}  + \frac12 y x^{t} &   y   \\
   \hdashline
        &     x^t      &   1   \\
   \end{array}
 \right).
\end{equation}
Then by a straightforward calculation, one checks that
the matrix
$$
A =  w\big( x , y, z\big)^{-1}u\big ((1- \textstyle{\frac12}z)x, y,z\big )
$$
satisfies $SAS^{-1} = A,$
hence belongs to $H.$ Thus, if $z \neq 2,$ then
\begin{eqnarray*}
a_s \, u(x , y , z) \cdot H
& = &
 a^s
 w\big((1- \textstyle{\frac12} z)^{-1} x , y, z\big) \cdot H
 \\
 &=&
  a^s  \,w\big((1 - \textstyle{\frac12} z)^{-1} x , y, z\big) \,
n_1\, a_{-s} \, \kappa^{-1} \, q(s)  \cdot H\\
&=&
a^s  \, w\big((1 - \textstyle{\frac12} z)^{-1} x , y, z + 1\big) a_{s}^{-1} \,  \kappa^{-1} \, q(s)  \cdot H \\
&=&
w\big((1 - \textstyle{\frac12} z)^{-1} e^{-s} x , e^{s} y, e^{2s}(z + 1)\big)\, \kappa^{-1} \, q(s) \cdot H.
\end{eqnarray*}
Through the substitution (\ref{e: substitution x y z}) the last expression becomes
$$
w\big(2(3- e^{-2s} \omega )^{-1} \xi , \eta, \omega\big) \; \kappa^{-1} \; q(s) \cdot H.
$$
For $s \to \infty$ this expression tends to
$$
w({\textstyle\frac23 }\xi, \eta, \omega) \, \kappa^{-1} \cdot H = \kappa^{-1} v_{(\omega, \eta), \frac23 \xi } \cdot H,
$$
see (\ref{e: defi w}) and  (\ref{eq def v_(x,y)}).
The result follows.
\end{proof}

%%% ----------------------------------------------------------------------
 \bibliographystyle{plain}
 %\bibliography{Bibliography}
 \newcommand{\SortNoop}[1]{}\def\dbar{\leavevmode\hbox to 0pt{\hskip.2ex
  \accent"16\hss}d}

%%%-----------------------------------------------------------------------
\def\adritem#1{\hbox{\small #1}}
\def\distance{\hbox{\hspace{0.5cm}}}
\def\apetail{@}
\def\addVdBan{\vbox{
\adritem{E.~P.~van den Ban}
\adritem{Mathematical Institute}
\adritem{Utrecht University}
\adritem{PO Box 80 010}
\adritem{3508 TA Utrecht}
\adritem{The Netherlands}
\adritem{E-mail: E.P.vandenBan{\apetail}uu.nl}
}
}
\def\addKuit{\vbox{
\adritem{J.~J.~Kuit}
\adritem{Dep. of Mathematical Sciences}
\adritem{University of Copenhagen}
\adritem{Universitetsparken 5}
\adritem{2100 K\o benhavn \O}
\adritem{Denmark}
\adritem{E-mail: j.j.kuit{\apetail}gmail.com}
}
}
\def\addSchlichtkrull{\vbox{
\adritem{H.~Schlichtkrull}
\adritem{Dep. of Mathematical Sciences}
\adritem{University of Copenhagen}
\adritem{Universitetsparken 5}
\adritem{2100 K\o benhavn \O}
\adritem{Denmark}
\adritem{E-mail: schlicht{\apetail}math.ku.dk}
}
}
\mbox{}
\vfill
\hbox{\vbox{\addVdBan}\vbox{\distance}\vbox{\addKuit}\vbox{\distance}\vbox{\addSchlichtkrull}}

%%% ----------------------------------------------------------------------
\end{document}